\documentclass[a4paper, 12pt]{amsart}
\usepackage{amsmath, amsthm, amscd, amssymb, amsfonts, amsxtra, latexsym}
\usepackage{enumerate}
\usepackage{verbatim}
\usepackage{float}

\usepackage{hyperref}
\hypersetup{colorlinks = true,	allcolors  = blue}

\newcommand{\Z}{\mathbb{Z}}
\newcommand{\N}{\mathbb{N}}
\newcommand{\ff}{\mathbb{F}}
\newcommand{\Tr}{\operatorname{Tr}}

\newcommand{\Cay}{\mathrm{Cay}}

\newcommand{\G}{\Gamma}

\newcommand{\tcbinom}[2]{\genfrac{[}{]}{0pt}{1}{#1}{#2}}

\newcommand{\sk}{\smallskip}
\newcommand{\msk}{\medskip}

\newtheorem{thm}{Theorem}[section]
\newtheorem{prop}[thm]{Proposition}
\newtheorem{lem}[thm]{Lemma}
\newtheorem{coro}[thm]{Corollary}

\theoremstyle{definition}
\newtheorem{rem}[thm]{Remark}
\newtheorem{exam}[thm]{Example}
\newtheorem{defi}[thm]{Definition}

\theoremstyle{remark}

\usepackage{tikz}
\usetikzlibrary{arrows.meta}

\usepackage{xcolor}

\def\purple{\color{purple}}

\begin{document} \sloppy
\numberwithin{equation}{section}
\title[Solutions of systems of diagonal equations through GP-graphs]
{On the number of solutions of systems \\ of diagonal equations through diagonal GP-graphs: \\ the general and the Hermitian-form cases} 
\author{Ricardo A.\@ Podest\'a, Denis E.\@ Videla}
\dedicatory{\today}
\thanks{Partially supported by CONICET, FonCyT and SECyT-UNC}
\address{Ricardo A.\@ Podest\'a. FaMAF -- CIEM (CONICET), Universidad Nacional de C\'ordoba. 
	\newline Av.\@ Medina Allende 2144, Ciudad Universitaria, (5000), C\'ordoba, Argentina. \newline
	{\it E-mail: podesta@famaf.unc.edu.ar}}
\address{Denis E.\@ Videla. FaMAF -- CIEM (CONICET), Universidad Nacional de C\'ordoba. 
	\newline	Av.\@ Medina Allende 2144, Ciudad Universitaria, (5000), C\'ordoba, Argentina.
	\newline {\it E-mail: devidela@famaf.unc.edu.ar}}

\begin{abstract}
For any $m, s \in \N$, we study the number $N_{m\times s,q}(\kappa, \beta)$ of solutions $(x_1,\ldots,x_s) \in (\ff_q)^s$ of the monic system of diagonal equations  
	$$ X_{1}^{k_i} + \cdots + X_{s}^{k_i}= \beta_i, \qquad (1\le i \le m), $$ 
with $\kappa=(k_1,\ldots,k_m) \in \N^m$ and $\beta=(\beta_1,\ldots,\beta_m) \in (\ff_q)^m$. 
We show that this number can be obtained 
in terms of some data of \textit{diagonal} GP-graphs $\G(\kappa,q)$. This is a new family of graphs that we introduce here, 
$\G(\kappa,q)$, with $\kappa = (k_1,\ldots,k_m) \in \N^{m}$, is the directed graph with vertex set the finite field $\ff_q$ and there is an arc from $u$ to $v$ if and only if $v-u \in R_{\kappa} = \{ (x^{k_1},\ldots,x^{k_m}) : x \in \ff_{q}^*\}$.
In particular, we give three different expressions for $N_{m\times s,q}(\kappa, \beta)$: one in terms of walks, another in terms of adjacency matrices of $\G(\kappa,q)$ and the last one in terms of the spectrum of $\G(\kappa,q)$. 
Finally, we explicitly derive combinatorial formulas for the number of solutions $N_{m}(s,q) = N_{m\times s,q}(\kappa_\ell, 0)$ of monic homogeneous systems of diagonal equations of the form 
$$ X_1^{q^{\ell_i}+1} + \cdots + X_s^{q^{\ell_i}+1} = 0 \qquad (1\le i \le m),$$ 
with $\kappa_\ell=(\ell_1,\ldots,\ell_m)=(1,3,\ldots,2m-1)$ and $m\ge 2$,
via the known spectrum of Hermitian-form graphs, which can be viewed as diagonal GP-graphs.
For any $m,s \in \N$, we give general summation and recursive formulas for $N_m(s,q) \in \Z[q]$. 
For the small cases $N_{1}(s,q)$, $N_{2}(s,q)$ and $N_{m}(s,q)$, with $1\le s \le 5$, we give explicit expressions.
\end{abstract}
	
\maketitle
	
\vspace{-1.9em}	
\section{Introduction} \label{sec: Intro}
In this work, we study the exact number of solutions of  certain systems of monic diagonal equations over finite fields.
We will do this by relating the data of the system with some diagonal GP-graphs that we define. 
In particular, the number of solutions of the systems will be given by the spectrum of the graphs. 
The case of one equation will be treated separately in a companion work.

\subsection*{Systems of diagonal equations}
Let $\ff_{q}$ be a finite field with $q$ elements, where $q$ is a prime power and $m, s \in \N$.
We will study particular instances of the system of diagonal equations (SDEs for short) over $\ff_{q}$ 
\begin{equation} \label{eq: gen SDE}
	\mathcal{S}_{m \times s,q}(A, \kappa, \beta) := \quad \begin{cases}
		\begin{tabular}{ccccccc}	
			$\alpha_{11} X_1^{k_1}$ & $+$ & $\cdots$ & $+$ & $\alpha_{1s} X_s^{k_1}$  & $=$ & $\beta_1$, \\
			$\vdots$          &       & $\vdots$ & & $\vdots$  &  & $\vdots$ \\
			$\alpha_{m1} X_1^{k_m}$ & $+$ & $\cdots$ & $+$ & $\alpha_{ms} X_s^{k_m}$  & $=$ & $\beta_m$.   
		\end{tabular}
	\end{cases}
\end{equation}
where  $A=(\alpha_{ij}) \in M_{m\times s} (\ff_q)$ is an $m \times s$ matrix with coefficients in the finite field $\ff_q$, 
$\beta = (\beta_1,\ldots,\beta_m) \in (\ff_q)^m$ is an $m$-tuple of field elements in $\ff_q$ and 
\begin{equation} \label{eq: vector K}
	\kappa = (k_1,\ldots,k_m) \in \N^m
\end{equation} 
is an $m$-tuple of positive integers. 
Without loss of generality, we will assume throughout the paper that $k_1\le k_2 \le \cdots \le k_m$.
 
If $m,s$ and $\ff_q$ are understood, we can write the system $\mathcal{S}_{m \times s,q}(A, \kappa, \beta)$ in matrix notation simply as
	$$ AX^\kappa = \beta, $$
where we use the notation $X^\kappa= (X_1^{k_1}, \ldots, X_s^{k_s})$ with $X=(X_1, \ldots, X_s)$ and $X$ and $\beta$ are thought as column vectors.

We will denote by $N_{m \times s,q}(A, \kappa, \beta)$ the number of solutions $x=(x_1,\ldots,x_s)$ in $(\ff_q)^s$ of the system \eqref{eq: gen SDE}, that is 
\begin{equation} \label{eq: No sols SDE}
	N_{m \times s, q}(A, \kappa, \beta) = \# \big\{ 
	x \in (\ff_q)^s \: : \: A_i \cdot x^\kappa = \beta_i, \: 1 \le i \le m
	\big\},
\end{equation}	
where $A_i$ denotes the $i$-th row of $A$ and $\cdot$ is the scalar product of vectors in $(\ff_q)^s$.
The same system can have different number of solutions in a subfield $\ff_t$ or in an extension $\ff_r$ of $\ff_q$. When we want to say something about the field where the solutions are considered we add the field in the notation. Thus, we will write 
$N_{m \times s, q}^{\ff_r}(A, \kappa, \beta)$ to denote the number of solutions of the system \eqref{eq: gen SDE} in $\ff_r$.
	
When $A=J$, the all 1's matrix, we will say that the system is \textit{monic} and we drop $A$ from the notation, i.e.\@ 
$\mathcal{S}_{m \times s, q}(\kappa, \beta) = \mathcal{S}_{m \times s, q}(J, \kappa, \beta) $. 
In this case we have to assume further that $k_1< \cdots < k_m$ (for if not, either we have repeated equations or inconsistent equations).
The system is \textit{homogeneous} if $\beta=(0,\ldots,0) \in (\ff_q)^m$, and we denote it simply by $\mathcal{S}_{m \times s, q}(A, \kappa)$. Hence, if the system is monic and homogeneous (which is the case we are mainly interested in) we just write 
	$$ \mathcal{S}_{m \times s, q}(\kappa). $$
Similarly, we make the same abbreviations for the number of solutions $N_{m \times s, q}(A, \kappa, \beta)$ of the corresponding system.  
	
\subsection*{A short history} 
There are only a few works about the explicit computation of the number of solutions of this kind of systems in the literature. 
Recently, by using algebraic geometry tools, Perez and Privitelli (\cite{PPri}, \cite{PPri2}) gave estimates for the number of solutions of certain systems of diagonal equations over finite fields. See the works \cite{ZHJYC} and \cite{ZWZH} for the relation of the number of solutions of (1.1) with the weight distribution of reducible cyclic codes and \cite{MC2} for the relation with the covering number of cyclic codes.

Beyond bounds and congruences, exact evaluations have been historically confined to specific families. For two equations we can only refer to three works, those of Carlitz, Tiet\"av\"ainen and Cao-Chou-Gu.
For instance, systems of two quadratic forms over $\ff_q$  
$$ \begin{cases}
 Q_1(X) = \sum a_j X_j^2 = b_1, \\[1.5mm]
 Q_2(X) = \sum c_j X_j^2=b_2, 	
\end{cases}$$
with $q$ odd were successfully enumerated by Carlitz \cite{Carlitz1953} using the invariant theory of pencils of quadrics. 
For higher degrees, Tiet\"av\"ainen in 1965 (see \cite{Tietavainen1965}) found explicit formulas for symmetric diagonal systems  of $m$ equations in $m \cdot k$ variables:
	$$	
	\sum_{1\le j \le k} X_{1,j}^d = 0, \qquad \ldots, \qquad  \sum_{1\le j \le k} X_{m,j}^d = 0, $$
for $d\mid q-1$, under strict divisibility conditions. 
More recently, in 2016, Cao, Chou, and Gu (\cite{CCG}) successfully used classical character sum methods to systems of two monic homogeneous diagonal equations of the form 
$$
\begin{cases}
\hfil	X_1 + \cdots + X_s & =0, \\[1.5mm] 
	X_1^{p^{\ell}+1} + \cdots +  X_s^{p^{\ell}+1} & = 0, 	
\end{cases}
$$ 
where $s\ge 2$, $\ell\ge 1$ and $q=p^{2\ell t}$ for $p$ prime and $t \in \N$, to found the number of solutions.
That is, only for the systems with $m=2$ and $\kappa=(2,2)$, any $m$ and $\kappa=(d,\ldots,d)$, and $m=2$ and $\kappa=(1,p^\ell+1)$, under certain conditions, the number of solutions can be explicitly computed.

On the other hand, the case of one equation (i.e. $m = 1$) has been extensively studied, primarily using classical number-theoretic approaches. The interested reader can refer to the pioneering work of Weil (\cite{W}), which relates the number of solutions in terms Jacobi and Gauss sums (see also the mandatory reference books \cite{LN}, \cite{MP}), as well as to the very recent advances by Oliveira \cite{Oliveira2026} on Fermat-type equation. The second author have studied the number of solutions of a particular case of diagonal equation by using GP-graphs with NEPS structure (see \cite{V}). 

In contrast to these classical techniques, our work introduces a novel graph-theoretic perspective to compute the exact number of solutions for systems with an arbitrary number of equations ($m \ge 2$). We treat the $m=1$ case with our techniques in a complementary forthcoming work \cite{PV26}.

\subsection*{New family of graphs: diagonal GP-graphs}
If $G$ is an abelian group and $S$ a subset of $G$ not containing $0$, 
the associated Cayley graph 
	$ \Gamma = \Cay(G,S) $ 
is the directed graph (digraph) with vertex set $G$ and where two vertices $u,v$ form a directed edge from $v$ to $u$ in $\Gamma$ if and only if $v-u \in S$. 
We will also need to consider the reversed Cayley graph $\overleftarrow{\G}$, the graph obtained from $\G$ by reversing all of its arcs. In general we have that 
	$$ \overleftarrow{\Cay}(G,S)=\Cay(G,-S) \qquad \text{and} \qquad \Cay(G,S) \simeq \Cay(G,-S).$$
If $S$ is symmetric ($S=-S$), then $\Cay(G,S)$ is a simple (undirected without loops) graph (considering each pair of arcs $\vec{xy}$ and $\vec{yx}$ as a single undirected edge $xy$), and hence $\Cay(G,S)=\overleftarrow{\Cay}(G,S)=\Cay(G,-S)$.

Given a finite field $\ff_q$ and $m\in \N$, for each $m$-tuple $\kappa = (k_1,\ldots,k_m) \in \N^{m}$ we define the 
\textit{diagonal generalized Paley graphs} (\textit{diagonal GP-graphs} for short), 
which are the reversed Cayley graphs 
\begin{equation} \label{eq: diag GPs} 
	\G(\kappa ,q) = \Cay(\ff^{m}_{q}, -R_\kappa) \qquad \text{with} \qquad 
		R_\kappa  = \{ (x^{k_1}, \ldots, x^{k_m}) : x \in \ff_{q}^* \}.  
\end{equation} 
That is, $\G(\kappa ,q)$ is the graph with vertex set $V=(\ff_q)^m$, such that for the vectors $u = (u_1,\ldots,u_m)$, $v=(v_1,\ldots,v_m) \in \ff^{m}_{q}$
there is an arrow from $u$ to $v$ if and only if 
there is some $x \in \ff_{q}^*$ such that 
	$$ v_i-u_i = x^{k_i} $$ 
for all $i=1,\ldots,m$.

One interesting instance of these graphs is when $m=1$, 
this graph is called \textit{generalized Paley graph} (\textit{GP-graph} for short) and is denoted by $\G(k,q)$,
more precisely 
\begin{equation*} \label{Gammas GP}
	\G(k,q) = Cay(\ff_{q}, -R_{k}) \qquad \text{with} \qquad R_k = \{ x^{k} : x \in \ff_{q}^*\}.
\end{equation*}
In this case, since 
\begin{equation} \label{eq: GKq = Gk'q}
	\G(k,q)=\G(k',q) \qquad \text{where} \qquad k'=\gcd(k,q-1),
\end{equation}
it is common to assume that $k\mid q-1$. 
Notice that $\G(k,q)$ is an $n$-regular graph with $n=\tfrac{q-1}k$. 
The graph $\G(k,q)$ is undirected either if $q$ is even or if $k$ divides $\tfrac{q-1}2$ when $p$ is odd (equivalently if $n$ is even when $p$ is odd) 
and it is connected if $n$ is a primitive divisor of $q-1$.  
The GP-graphs have been extensively studied in the few past years (see for instance \cite{LP}, \cite{PP}, \cite{PV6}, \cite{PV7}, \cite{PV4}, \cite{PV3}, \cite{PV1}, \cite{PV8}, \cite{PV19}, \cite{V}, \cite{Yip}).

\subsection*{Outline and results}
We now discuss the structure of the paper. 
Briefly, in Section~\ref{sec: diagonal GPs} we introduce diagonal GP-graphs, in Sections \ref{sec: SDE and walks}--\ref{sec: Homogeneous SDEs and spectrum} we study the number of solutions of general systems of monic diagonal equations over finite fields while in Sections \ref{sec: Hermitian}--\ref{sec: hypergeometric} we study the number of solutions of the particular case of Hermitian-form systems of diagonal equations.
	
More precisely, 
in Section \ref{sec: diagonal GPs}, we study the basic structure of the class of diagonal GP-graphs defined above in \eqref{eq: diag GPs}.
In Theorem \ref{teo: Spec G(K,q)} we give the spectrum of the graphs $\G(\kappa, q)$ in terms of Weil sums. Namely, the eigenvalues are given by 
	\begin{equation}\label{eq: eigenvalues}
	\lambda_{\alpha,\kappa,q}  = \chi_\alpha(R_\kappa) = \sum_{(x_{1},\ldots, x_{m}) \in R_{\kappa}} \chi_{\alpha_{1}}(x_{1}) \cdots \chi_{\alpha_{m}}(x_{m}) =\tfrac{1}{d_\kappa} \big( \mathcal{W}_\chi(f_{\alpha,\kappa}) -1 \big) 
	\end{equation}
with $\alpha \in (\ff_q)^m$, 
	$$ 
		d_\kappa = \gcd(\kappa,q-1) = \gcd(k_1,\ldots,k_m,q-1), 
	$$ 
and where $\mathcal{W}_\chi(f_{\alpha,\kappa}) = \sum_{x \in \ff_q} \chi(f_{\alpha,\kappa}(x))$ is the Weil sum 
associated to the canonical additive character $\chi$ of $\ff_q$ and the polynomial 
	$$ f_{\alpha,\kappa}(x) = \alpha_1 x^{k_1} + \cdots + \alpha_m x^{k_m} \in \ff_q[x]. $$ 
As an application, in Proposition  \ref{prop: Ihara Zeta GP graph} we give the Ihara zeta function of the diagonal GP-graphs in terms of Weil sums.

\sk

In Section \ref{sec: SDE and walks}, we consider monic systems of diagonal equations, that is systems 
$\mathcal{S}_{m\times s, q}(\kappa, \beta)$ as in \eqref{eq: gen SDE} where $\alpha_{ij}=1$ for every $i$ and $j$. 
In Lemma \ref{lema: walks} we relate the number of solutions of $\mathcal{S}_{m \times s, q}(\kappa, \beta)$ with the number $w_\G(s;0,\beta)$ of $s$-walks from $0$ to $\beta$ of the graph $\G=\G(\kappa,q)$. 
In Proposition~\ref{prop: recursion of Nmxsq} we give recursive formulas for  
the number of solutions $\mathcal{S}_{m \times s, q}(\kappa, w-v)$ with $w,v \in (\ff_q^*)^\ell$ in terms of the number of solutions of the smaller systems $\mathcal{S}_{m \times \ell, q}(\kappa, w-v)$ for $1\le \ell \le s$.
Then, in Theorems \ref{thm: Nmsq from walks} and \ref{teo: Nmsq no-homogeneo con matrices} we give the number of solutions $N_{m \times s,q}(\kappa, \beta)$ of the system $\mathcal{S}_{m \times s, q}(\kappa, \beta)$ in two different ways:
$$ N_{m \times s,q}(\kappa, \beta) =
\delta_{\beta,0} + \sum\limits_{\ell=1}^{s} \tbinom{s}{\ell} \, d_{\kappa}^{\ell} \, w_{\G}(\ell;0,\beta) = \big( (I_{q^m}+d_{\kappa} A_{\G})^{s}\big)_{0,\beta},$$
where $A_{\G}$ is the adjacency matrix of $\G$.
Then, we show the relation between the number of solutions and the Ihara zeta function, via the number of closed walks. 
In fact, in Proposition \ref{prop: Ihara and solutions} we give the logarithmic expansion $\log \zeta_\Gamma(u)$ of the Ihara zeta function of $\Gamma=\G(\kappa,q)$, which is completely determined by the number of solutions $N_{m \times \ell, q}(\kappa)$.

\sk

Next, in Section \ref{sec: Homogeneous SDEs and spectrum} 
we study the number of solutions of monic SDEs in terms of the spectrum of the diagonal GP-graphs. 
In fact, in Theorem \ref{teo: Nmsq chars y spec} we obtain an expression for the number of solutions of a general system in terms of characters $\chi_\beta$ of $\ff_{q^m}$ and eigenvalues $\lambda_\alpha$ of $\G(\kappa,q)$ with $\alpha, \beta \in \ff_{q^m}$, while in Corollary \ref{coro: weil bound} we give an  expression for this number in terms of Weil sums; namely 
	$$ N_{m \times s,q}(\kappa, \beta) = \tfrac{1}{q^m} \sum_{\alpha \in \ff_{q^m}}  \overline{\chi_\beta(\alpha)} \, \big( 1+d_\kappa \lambda_\alpha \big)^s = \tfrac{1}{q^{m}} \sum_{\alpha \in \ff_q^m} \overline{\chi(\alpha \cdot \beta)} \big( \mathcal{W}_\chi (f_{\alpha,\kappa}) \big)^{s}.$$
From this, putting $k_1 \le \cdots \le k_m$ and using the Weil bound for Weil sums, we get the upper bound  
	$$	
	\left| N_{m \times s, q} (\kappa,\beta) - q^{s-m} \right| \le (q^m-1) q^{\frac s2 -m} (k_m-1)^s.
	$$
		
In the last four sections, we consider homogeneous systems of monic diagonal equations of a special kind; 
first we treat the general problem and then we deal with small systems in more detail. 
More precisely, in Section \ref{sec: Hermitian} we consider the systems $\mathcal{S}_{m\times s, q}(\kappa_{\ell})$ with 
$\kappa_{\ell}=(q+1,q^3+1,\ldots,q^{2m-1}+1)$, that is 
	$$ X_1^{q^{\ell_i}+1} + \cdots + X_s^{q^{\ell_i}+1} = 0 \qquad (1\le i \le m)$$ 
with $(\ell_1,\ldots,\ell_m)=(1,3,\ldots,2m-1)$. We call this the homogeneous monic Hermitian-form system.  
By using Theorem \ref{teo: M with Spec} and the known spectrum of Hermitian-form graphs, which turn out to be diagonal GP-graphs with $\G(\kappa_{\ell},q)$,
in Theorem \ref{teo: hermitiano diagonal} we obtain the number of solutions of $\mathcal{S}_{m\times s, q}(\kappa_{\ell})$ in $(\ff_{q^{4m}})^s$, which is given by 
\begin{equation} \label{eq: Nsm alternativa}
	N_{m \times s,q^{4m}}(\kappa_{\ell}) = q^{-2m(2m-2s)} \Bigg( 1+\sum_{j=1}^{2m}  (-q)^{-js} \, \prod_{i=0}^{j-1} \frac{q^{4m}-q^{2i}}{q^{j}-(-1)^{i+j}q^{i}} \Bigg).
\end{equation}	
In what follows, we will use the abbreviation $N_m(s,q)$ for $N_{m \times s,q^{4m}}(\kappa_{\ell})$.
As a direct application, we give the Ihara zeta function of Hermitian-form graphs (see Corollary \ref{cor: Ihara Hermitian}).

Then, in Section \ref{sec: Small HF-SDE} we consider the previous system of monic homogeneous Hermitian-form diagonal equations with one, two or three equations. In all cases we give summation formulas for the number of solutions (see Propositions \ref{prop: N_1xs}, \ref{prop: N_2xs}, \ref{prop: N_3xs}) and some examples. 
Moreover, for $m=1,2$ we give closed formulas. For instance, in Proposition \ref{prop: N_1xs} we obtain
that the number of solutions in $(\ff_{q^4})^s$ of the Hermitian-form diagonal equation $X_1^{q+1} + \cdots + X_s^{q+1}=0$ 
is given by
$$	N_1(s,q) = N_{1 \times s,q^4} (q+1) = q^{2s-3} \big\{ q^{2s-1} + (q-1) (q^2+1) \big((-1)^{s}q^{s-1} +1 \big)  \big\}.$$
In Proposition \ref{prop: N_2xs} we give the corresponding number $N_2(s,q)$ for homogeneous systems of two monic Hermitian-form diagonal equations, the obtained expression is more involved.

Then, in Section \ref{sec: caso gral}, we deal with the dual case, that is any number $m$ of equations but with the number $s$ of variables fixed. We have that $N_m(1,q)=1$ trivially and in Proposition~\ref{prop: Nm2q Nm3q} we show that the exact number of solutions are
$$	N_m(2,q) = q(q^{4m} + q^{4m-1} - 1) \qquad \text{and} \qquad 
N_m(3,q) = q^3(q^{4m} + q^{4m-3} - 1).$$	  
We then obtain expressions for $N_m(s,q)$ for $1\le s \le 12$.	
To conclude the section, we come to two main results. In Theorem \ref{thm: recursive Nmsq} we obtain a two-term recursion for the numbers $N_m(s,q)$ for any $m\ge 1$ and any $s\ge 2$, namely
$$ N_m(s,q) = (-q)^{s-1} N_m(s-1,q) + q^{4m} \big\{1 - (-q)^{s-1} \big\} N_m(s-2,q), $$
with initial conditions $N_m(0,q) = 1$ and $N_m(1,q) = 1$.
Moreover, in Theorem \ref{thm: sums for Nmsq} we show that for any $m, s \in \N$, the number $N_m(s,q)$ is given by 
	$$	N_m(s,q) = \Big( \sum_{j=0}^{k-1} (-1)^j q^{j(2j-(-1)^k)} q^{4m(k-j)} \tcbinom{k}{j}_{q^2} \mathcal{F}_k \Big) + (-1)^k q^{k(2k-(-1)^k)} $$
where $k=\lfloor \tfrac s2 \rfloor$, the symbol ${k \brack j}_{q^2}$ stands for the $q^2$-binomial coefficients, $F_{\ell}=q^{2\ell-1}+1$ and  
$\mathcal{F}_k = F_{j+1} \cdots F_k$ if $s$ is even and $\mathcal{F}_k = F_{j+2} \cdots F_k$ if $s$ is odd, thus obtaining the more general expression for $N_m(s,q)$ for any $m,s,q$.

Finally, in Section \ref{sec: hypergeometric} we study the hypergeometric connection of the problem of finding the number of solutions of the monic homogeneous Hermitian-form SDEs.
In Theorem~\ref{prop: hipergeometrica} we give the surprising expression 
$$ N_m(s,q) = q^{4m(s-m)} \ _2\phi_0 \left( \begin{matrix} q^{-2m}, -q^{-2m} \\ - \end{matrix} ; -q, -(-q)^{4m-s} \right)$$
in terms of a Heine $q$-hypergeometric series of the form $_2\phi_0 \left( \begin{smallmatrix} a, b \\ - \end{smallmatrix}; q, z \right)$.

\section{Diagonal generalized Paley graphs} \label{sec: diagonal GPs}
In this section, we give the basic structural properties of the diagonal GP-graphs $\G(\kappa,q)$ and obtain a formula for the eigenvalues in terms of Weil sums $W_\chi(f_{\alpha,\kappa})$ for certain polynomials $f_{\alpha,\kappa}(x) \in \ff_q[x]$ with $\alpha \in (\ff_q)^m$.

We begin by setting some notations. 
Given a vector of positive integers
\begin{equation} \label{eq: vector kappa}
	\kappa=(k_1,\ldots,k_m)\in \N^m
\end{equation} 
and $q$ a prime power, we define the number 
\begin{equation} \label{eq: k hat}
	d_\kappa = \gcd(\kappa,q-1) = \gcd(k_1,\ldots,k_m,q-1)
\end{equation}
and the set 
	\begin{equation}\label{eq: Uk}
		U_\kappa = \{x\in \ff_{q}^*:\, (x^{k_1},\ldots,x^{k_m})=(1,\ldots,1)\}.
	\end{equation}
	
The following lemma gives the sizes of $R_\kappa$ and $U_\kappa$, and hence the degree of regularity of the graphs $\G(\kappa,q)$, which we will need in the proof of the next result.

\begin{lem} \label{lem: size Rk} 
Let $\kappa = (k_1,\ldots,k_m)\in \N^{m}$ and let $q$ be a prime power.
If $R_{\kappa}$, $d_\kappa$ and $U_{\kappa}$ are as in \eqref{eq: diag GPs}, \eqref{eq: k hat} and \eqref{eq: Uk} respectively, then we have that 
	$$ |R_\kappa| = \tfrac{q-1}{d_\kappa} \qquad \text{and} \qquad |U_\kappa| = d_\kappa.$$
\end{lem}
	
\begin{proof}
Let $\omega$ be a primitive element of $\ff_{q}$ and notice that in this case we have that 
	$$ R_{\kappa} = \langle (\omega^{k_1},\ldots,\omega^{k_m})\rangle. $$
So, it is enough to compute the multiplicative order of the element 
	$(\omega^{k_1},\ldots,\omega^{k_m})$ in the ring $(\ff_{q})^m$.
First, observe that if we put $t = \frac{q-1}{d_{\kappa}}$, 
then we have  
	$$(\omega^{k_1},\ldots,\omega^{k_m})^t = (1,\ldots,1) $$ 
since $n_i = \frac{q-1}{\gcd(k_i,q-1)} \mid t$ and $\omega^{n_i k_i}=1$ for all $i=1,\ldots,m$. 
		
On the other hand,
if $(\omega^{k_1},\ldots,\omega^{k_m})^{\ell}=(1,\ldots,1)$, then $\omega^{\ell k_i} =1$ for all $i=1,\ldots,m$.
Thus, we obtain that $q-1\mid \ell k_i$  and hence $\frac{q-1}{\ell}\mid k_i$ for all $i=1,\ldots,m$.
By definition of the greatest common divisor, we have that
	$\frac{q-1}{\ell}\mid \gcd(k_1,\ldots,k_m,q-1) = \tfrac{q-1}{t}$
which implies that
	$t \mid \ell$ and so we have that
	$$ord_{(\ff_q)^m}(\omega^{k_1},\ldots,\omega^{k_m}) = t = \tfrac{q-1}{\gcd(k_1,\ldots,k_m,q-1)}.$$
Therefore, $|R_{\kappa}| = \frac{q-1}{d_{\kappa}}$ as asserted.
As a direct consequence, the diagonal Cayley graph $\G(\kappa,q)$ has regularity degree equal to $\frac{q-1}{d_{\kappa}}$, as desired.
		
Finally, if $\omega$ is a primitive element of $\ff_{q}$, the order of $(\omega^{k_1},\ldots\omega^{k_{m}})$ is $\frac{q-1}{d_{\kappa}}$ and hence
	$$ U_{\kappa}= \{ \omega^{t \frac{q-1}{d_\kappa}}: t=1,\ldots, \kappa\}. $$
Therefore $|U_{\kappa}|=d_{\kappa}$ as desired.
\end{proof}
	
We now give the basic structural properties of these graphs. We will need the following notation. An integer $k$ is called a \textit{primitive divisor} of $p^r-1$, with $p$ prime and $r \in \N$, if $k\mid p^r-1$ but $k \nmid p^a-1$ for any $1\le a \le r-1$. This is denoted 
\begin{equation} \label{eq: primitive divisor}
	k \dagger p^r-1.
\end{equation}

\begin{prop} \label{prop: diag GPs}
	Consider the diagonal GP-graph $\G(\kappa,q)$ with $\kappa=(k_1,\ldots,k_m)\in \N^m$. Then, $\G(\kappa,q)$ is a loopless $\frac{q-1}{d_\kappa}$-regular graph. Moreover, 
	\begin{enumerate}[$(a)$]
		\item $\Gamma(\kappa,q)$ is undirected
			if and only if either $q$ is even, or $q$ is odd with $\tfrac{q-1}{d_\kappa}$ even and $\tfrac{k_i}{d_\kappa}$ odd for every $i=1,\ldots,m$. \sk 
		
		\item If $\G(\kappa,q)$ is connected then $\G(k_i,q)$ is connected for any $i=1,\ldots,m$; that is we have  $\frac{q-1}{\gcd(k_i,q-1)} \dagger q-1$ for any $i=1,\ldots,m$.
	\end{enumerate}
\end{prop}

\begin{proof}
	The graph $\G(\kappa,q)$ is $\frac{q-1}{d_\kappa}$-regular, by Lemma \ref{lem: size Rk}, and it has no loops because $0 \notin U_{k_i}$ for any $i=1,\ldots,m$.
	
	\noindent ($a$) 
	Recall that $\Gamma(\kappa,q)$ is undirected if and only if
		$R_\kappa=-R_\kappa$. If $q$ is even, this condition holds automatically.
		
		Suppose that $q$ is odd. Under coordinatewise multiplication,
		$R_\kappa$ is a cyclic subgroup of $(\mathbb{F}_q^*)^m$.
		Since $(1,\ldots,1)\in R_\kappa$, we have that
		$$
		R_\kappa=-R_\kappa
		\quad\Leftrightarrow\quad
		(-1,\ldots,-1)\in R_\kappa.
		$$
		Let $\omega$ be a primitive element of $\mathbb{F}_q$,
		then
		$$
		R_\kappa=\langle g\rangle,
		\qquad
		g=(\omega^{k_1},\ldots,\omega^{k_m}),
		\qquad
		\operatorname{ord}(g)=h:=\tfrac{q-1}{d_\kappa}.
		$$
		Since $(-1,\ldots,-1)$ has order two, it belongs to
		$R_{\kappa}$ if and only if $h$ is even and  $g^{h/2}=(-1,\ldots,-1)$. 
		Indeed, when $h$ is even, $g^{h/2}$ is the unique
		element of order two in $\langle g\rangle$,
		moreover
		$$
		g^{h/2}
		=
		\Big(
		(\omega^{(q-1)/2})^{k_1/d_\kappa},
		\ldots,
		(\omega^{(q-1)/2})^{k_m/d_\kappa}
		\Big)
		=
		\left(
		(-1)^{k_1/d_\kappa},\ldots,
		(-1)^{k_m/d_\kappa}
		\right).
		$$
		Thus $g^{h/2}=(-1,\ldots,-1)$ if and only if
		$k_i/d_\kappa$ is odd for every $i$, proving the claim.
	
	\noindent ($b$) The connectedness of $\G(\kappa,q)$ implies the connectedness of $\G(k_i,q)$ for any $i=1,\ldots,m$ and a GP-graph $\G(k,q)$ is connected if and only if $\frac{q-1}{\gcd(k,q-1)}\dagger q-1$.
\end{proof}

\subsection*{Weil sums and spectrum}
We recall that given an additive character $\chi$ of $\ff_q$ and a polynomial $f\in \ff_q[x]$, the \textit{Weil sum} associated to $\chi$ and $f$ is defined by 
\begin{equation} \label{eq: Weil sum}
	\mathcal{W}_\chi(f) = \sum_{a \in \ff_q} \chi(f(a)).
\end{equation}
Consider the \textit{canonical} additive character $\chi: \ff_q \rightarrow \mathbb{S}^1$ of $\ff_q$ given by
\begin{equation} \label{eq: canonical character}
	\chi(x) = e^{\frac{2\pi i}{p} \Tr_{q/p}(x)},
\end{equation} 
where $\Tr_{q/p} : \ff_q \rightarrow \ff_p$ is the trace map given by 
	$$ \Tr_{q/p}(x) = x+x^p+x^{p^2}+\cdots +x^{p^{r-1}}. $$

The following result states that the eigenvalues of this kind of graphs can be put in terms of Weil sums in a simple way.
	
\begin{thm} \label{teo: Spec G(K,q)} 
Let $\kappa = (k_1,\ldots,k_m)\in \N^{m}$ and $q$ a prime power.
Then, the eigenvalues of the diagonal GP-graph $\G(\kappa,q)$ are parametrized by the vectors 
$ \alpha = (\alpha_1, \ldots, \alpha_m) \in (\ff_{q})^m$ and are of the form  
\begin{equation} \label{eq: autovalores de G(kapkka,q)}
	 \lambda_{\alpha,\kappa,q} = \tfrac{1}{d_\kappa} \big( \mathcal{W}_\chi(f_{\alpha,\kappa}) -1 \big)  
\end{equation}	
where 
$\chi$ is the canonical additive character of $\ff_q$ and $f_{\alpha,\kappa}(x) = \alpha_1 x^{k_1} + \cdots + \alpha_m x^{k_m} \in \ff_q[x]$.
\end{thm}

\begin{proof}
Recall that for an abelian group $G$, the eigenvalues of the Cayley graph $Cay(G,S)$ are given by 
	$$ \lambda_\chi = \chi(S) = \sum_{s\in S} \chi(s) $$
where $\chi$ is a character of the group $G$ (see for instance \S2 in \cite{LZ}).

Notice that the characters $\chi_\alpha$, with $\alpha= (\alpha_{1},\ldots, \alpha_m) \in (\ff_q)^m$, of the abelian group $(\ff_q)^{m}$ are given by the products 
	$$
		\chi_{\alpha}(x) = (\chi_{\alpha_{1}}\cdots \chi_{\alpha_{m}})(x_1,\ldots,x_m) = \chi_{\alpha_1}(x_1) \cdots \chi_{\alpha_{m}}(x_m)
	$$
for any $x=(x_1,\ldots,x_m) \in (\ff_q)^m$ and where for any $\alpha_j \in \ff_{q}$, $\chi_{\alpha_j}$ is the additive character of $\ff_q$ associated to $\alpha_j$ given by
	$$ \chi_{\alpha_j}(x) = e^{\frac{2\pi i}{p} \Tr_{q/p}(\alpha_j x)}.$$

Since the connection set of the Cayley graph $\G(\kappa,q)$ is $R_{\kappa} = \{(x^{k_1}\ldots,x^{k_m}): x\in \ff_{q}^*\}$, 
we obtain that the eigenvalues of the diagonal GP-graphs are given by
\begin{equation} \label{eq: eigenvalues GP characters}
	\lambda_{\alpha,\kappa,q} = \chi_\alpha(R_\kappa) = \sum_{(x_{1},\ldots, x_{m}) \in R_{\kappa}} \chi_{\alpha_{1}}(x_{1}) \cdots \chi_{\alpha_{m}}(x_{m}).	
\end{equation}
Now, by the definition of the sets $R_{\kappa}$ and $U_k$ (see \eqref{eq: diag GPs} and \eqref{eq: Uk}), we have that 
	$$ (x^{k_1},\ldots, x^{k_s}) = (y^{k_1},\ldots, y^{k_s}) \quad \Leftrightarrow \quad 
		\big( (\tfrac{x}{y})^{k_1},\ldots,(\tfrac{x}{y})^{k_s} \big) = (1,\ldots,1) \quad \Leftrightarrow \quad
		\tfrac{x}{y} \in U_{\kappa}.$$
Thus, using this, Lemma~\ref{lem: size Rk} 
implies that 
	$$	\sum_{(x_{1},\ldots, x_{m}) \in R_{\kappa}} \chi_{\alpha_{1}}(x_{1}) \cdots \chi_{\alpha_{m}}(x_{m})  = \tfrac{1}{d_\kappa} \sum_{x\in \ff_{q}^*} \chi_{\alpha_1}(x^{k_1}) \cdots \chi_{\alpha_{m}}(x^{k_m}). $$   
That is, we have 
	$$ \lambda_{\alpha,\kappa,q} = \tfrac{1}{d_\kappa} \sum_{x\in \ff_{q}^*} \chi(\alpha_1 x^{k_1} + \cdots + \alpha_m x^{k_m}). $$
Finally, since 
$$	\mathcal{W}_\chi(f_{\alpha,\kappa}) = \sum_{x \in \ff_q} \chi(f_{\alpha,\kappa}(x)) = 1+  \sum_{x \in \ff_q^*} \chi(f_{\alpha,\kappa}(x))$$
we get \eqref{eq: autovalores de G(kapkka,q)}, as desired.
\end{proof}

\subsection*{Ihara zeta function}
We now relate the structural properties of diagonal GP-graphs to their Ihara zeta function. The Ihara zeta function of a graph is a discrete analogue of the Riemann zeta function, which encodes the lengths of all prime, non-backtracking, closed walks in the graph. 

Let $G = (V,E)$ be a finite connected undirected graph. A closed walk is called \textit{non-backtracking} if it does not contain any sequence of vertices of the form $v \to w \to v$. It is called \textit{prime} if it is not merely the repetition of a strictly shorter closed walk. Two prime walks are equivalent if they differ only by their starting vertex. The \textit{Ihara zeta function} of $G$ is formally defined as
\begin{equation} \label{eq: Ihara definition}
	\zeta_G(u) = \prod_{[C]} (1 - u^{L(C)})^{-1},
\end{equation}
where $[C]$ runs over all equivalence classes of prime, non-backtracking, closed walks in $G$, and $L(C)$ denotes the length of the cycle $C$. 
The graph-theoretic formulation of this zeta function for regular graphs was pioneered by Sunada \cite{Sunada1986}. For a comprehensive treatment and its generalization to irregular graphs, we refer the reader to the seminal work of Bass \cite{Bass1992} and the extensive monograph by Terras \cite{Terras2010}.

By the celebrated Ihara-Bass determinant formula (see \cite{Bass1992}), if $G$ is a $k$-regular graph with adjacency matrix $A$ 
its zeta function evaluates to the rational function
\begin{equation} \label{eq: Bass formula}
	\zeta_G(u)^{-1} = (1-u^2)^{|E|-|V|} \det \big( I - u A + u^2(k-1)I \big).
\end{equation}

By imposing the conditions for undirectedness obtained in Proposition 2.2, we can give a fully explicit expression for the Ihara zeta function of the diagonal GP-graphs in terms of Weil sums.

\begin{prop} \label{prop: Ihara Zeta GP graph}
Let $m \in \mathbb{N}$ and $\kappa \in \mathbb{N}^m$. Let $q$ be a prime power such that either $q$ is even, or $q$ is odd with $D=\frac{q-1}{d_{\kappa}}$ and $\frac{k_i}{d_\kappa}$ odd for all $i=1,\dots,m$. Then, the Ihara zeta function of the undirected diagonal GP-graph $\Gamma = \Gamma(\kappa, q)$ is given by
	\begin{equation} \label{eq: Ihara GP graph}
		\zeta_{\Gamma}(u) = \dfrac{1}{(1 - u^2)^{\rho_\Gamma} \prod\limits_{\alpha \in \mathbb{F}_q^m} \big( 1 - u \lambda_{\alpha, \kappa, q} + u^2(\frac{q-1}{d_\kappa} - 1) \big)},
	\end{equation}
where $d_\kappa = \gcd(\kappa, q-1)$, the exponent $\rho_\Gamma$ is given by
	$$
	\rho_\Gamma = \tfrac{q^m(D-2)}{2}  ,
	$$
and the eigenvalues $\lambda_{\alpha, \kappa, q}$ are given in terms of the Weil sums $\mathcal{W}_\chi(f_{\alpha, \kappa})$ as in \eqref{eq: autovalores de G(kapkka,q)}.
\end{prop}

\begin{proof}
Under the parity hypotheses on $q$ and $k_i$, ($a$) in Proposition \ref{prop: diag GPs} ensures that $\Gamma = \Gamma(\kappa, q)$ is an undirected graph. Furthermore, by Lemma \ref{lem: size Rk}, $\Gamma$ is $k$-regular with degree $k = \frac{q-1}{d_\kappa}$ and possesses $|V| = q^m$ vertices. 
	
We first compute the $|E| - |V|$. By the handshaking lemma, the number of undirected edges is $|E| = \frac{|V| k}{2} = \frac{q^m (q-1)}{2 d_\kappa}$. Therefore, we obtain
	$$
	|E|-|V|=\tfrac{q^m (q-1)}{2 d_\kappa} - q^m = \tfrac{q^m(D-2)}{2}= \rho_{\G}.
	$$
	
Now, we evaluate the characteristic determinant in the Ihara-Bass formula \eqref{eq: Bass formula}. 
Since the adjacency matrix $A_\Gamma$ is diagonalizable, the determinant of any polynomial evaluated at $A_\Gamma$ is simply the product of that polynomial evaluated at the eigenvalues of $A_\Gamma$. By Theorem 2.3, the eigenvalues of $A_\Gamma$ are parametrized by $\alpha \in \mathbb{F}_q^m$ and are explicitly given by $\lambda_{\alpha, \kappa, q} = \frac{1}{d_\kappa} \big( \mathcal{W}_\chi(f_{\alpha, \kappa}) - 1 \big)$. Thus,
	\[
	\det \big( I - u A_\Gamma + u^2(k-1)I \big) = \prod_{\alpha \in \mathbb{F}_q^m}\big( 1 - u \lambda_{\alpha, \kappa, q} + u^2 (\tfrac{q-1}{d_\kappa} - 1) \big).
	\]
Substituting $\rho_\Gamma$ and the expanded determinant back into \eqref{eq: Bass formula} yields the desired explicit expression for $\zeta_\Gamma(u)$.
\end{proof}

\section{The number of solutions of monic SDEs in terms of  $\G(\kappa,q)$} \label{sec: SDE and walks}
In this section, we consider systems of monic diagonal equations of the form  
	\begin{equation} \label{eq: monic SDE}
		\mathcal{S}_{m \times s, q} (\kappa, \beta) := \quad \begin{cases}
			\begin{tabular}{ccccccc}	
				$X_1^{k_1}$ & $+$ & $\cdots$ & $+$ & $X_s^{k_1}$  & $=$ & $\beta_1$, \\[1mm]
				$X_1^{k_2}$ & $+$ & $\cdots$ & $+$ & $X_s^{k_2}$  & $=$ & $\beta_2$, \\[1mm]
				$\vdots$    &     & $\vdots$ & 	   & $\vdots$  	  &  	& $\vdots$ \\[1mm]
				$X_1^{k_m}$ & $+$ & $\cdots$ & $+$ & $X_s^{k_m}$  & $=$ & $\beta_m$,   
			\end{tabular}
		\end{cases}
	\end{equation}
where $\kappa = (k_1,\ldots,k_m) \in \N^m$ and $\beta = (\beta_1,\ldots,\beta_m) \in (\ff_q)^m$.
We will give two expressions for the number of solutions of these systems over finite fields: one in terms of walks of associated diagonal GP-graph 
and the second one in terms of the adjacency matrix of the associated diagonal GP-graph.

We denote by $N_{m \times s, q}(\kappa, \beta)$ the number of solutions of this system (see \eqref{eq: No sols SDE}), that is 
\begin{equation} \label{eq: Nsmq monico}
	N_{m \times s, q}(\kappa, \beta) = \# \big\{(x_1,\ldots,x_s) \in (\ff_q)^s \: : \: x_1^{k_i}+\cdots+ x_s^{k_i} = \beta_i,  \:  1\le i \le m \big\},	
\end{equation}	
and by $N_{m \times s, q}^*(\kappa, \beta)$ the number of solutions of \eqref{eq: monic SDE} with non-zero coordinates, that is 
\begin{equation} \label{eq: Nsmq monico no0}
	N_{m \times s, q}^*(\kappa, \beta) = \# \big\{(x_1,\ldots,x_s) \in (\ff_q^*)^s \: : \: x_1^{k_i}+\cdots+ x_s^{k_i} = \beta_i,  \:  1\le i \le m \big\}.
\end{equation}

The case when the system in \eqref{eq: monic SDE} is homogeneous, i.e.\@ $\beta = (0,\ldots,0)$, will be studied in more detail in the next section, where we will give the number of solutions of $\mathcal{S}_{m \times s, q} (\kappa)$ by a formula depending on the spectrum of the graph $\G(\kappa,q)$.

\subsection{Walks in $\G(\kappa,q)$}
Given a graph $G$ and vertices $v,w$ of $G$, it is usual to denote by 
$w_{G}(s;v,w)$ the number of walks of length $s$ (or \textit{$s$-walks}) in $G$ from vertex $v$ to vertex $w$, 
that is 
\begin{equation} \label{eq: definition number of walks v-->w}
	w_{G}(s;v,w) = \#\{ v v_2 \cdots v_{s-1} w : v_2,\ldots,v_{s-1} \in V(G) \}.
\end{equation}
In particular, $w_G(s;v,v)$ is the number of closed $s$-walks from $v$ to $v$. 
By convention,
	$$ w_{G}(0;v_i,v_j)=0 \quad \text{if $v_i\neq v_j$}  \qquad \text{and} \qquad w_{G}(0;v_i,v_i)=1. $$ 
Hence, the total number $w_G(s)$ of $s$-walks in $G$ is given by 
\begin{equation} \label{eq: wG(s)}
	w_G(s) = \sum_{v,w \in V} w_G(s;v,w).
\end{equation}

Next, given $v,w \in (\ff_q)^m$, we relate the number $N^*_{m \times s, q}(\kappa, w-v)$ of solutions over $\ff_{q}$ with non-zero coordinates of the monic system of diagonal equations  $\mathcal{S}_{m \times s, q}(\kappa, b)$ in \eqref{eq: monic SDE} 
with $b=w-v$
and the number $w_{\G(\kappa,q)}(s;v,w)$ of $s$-walks from $v$ to $w$ in the graph $\G(\kappa,q)$.
	
\begin{lem} \label{lema: walks}
Let $s,m \in \N$ and $\kappa \in \N^{m}$.
If $v,w \in (\ff_q)^m$, then 
	\begin{equation} \label{eq: walk sol} 
		w_{\G(\kappa,q)}(s;v,w) = \tfrac{1}{(d_\kappa)^s} \, N_{m \times s, q}^*(\kappa, w-v).
	\end{equation}
\end{lem}

\begin{proof}
Let $\kappa=(k_1, \ldots,k_m) \in \N^{m}$. If $v,w\in (\ff_q)^m$ form an edge in the graph $\G(\kappa,q)$ then there is an $x \in \ff_q^*$ such that $v_i-w_i=x^{k_i}$ for all $i=1,\ldots,m$. 
Iterating this process, an $s$-walk from $v$ to $w$ in $\G(\kappa,q)$ gives $x_{1},\ldots,x_{s}\in \ff_{q}^*$ such that
\begin{equation} \label{eq: sums power}
	v_i+x_1^{k_i}+\cdots+x_s^{k_i}=w_i \qquad (1 \le i \le m).
\end{equation} 
Notice that, given $x\in \ff_{q}^*$, by Lemma \ref{lem: size Rk} there are exactly $d_\kappa$ elements $y\in \ff_{q}^*$ such that 
\begin{equation} \label{eq: vec x^k = vec y^k}
	(x^{k_1},\ldots, x^{k_m}) = (y^{k_1},\ldots,y^{k_m}).
\end{equation}
So, each walk induces a number $(d_\kappa)^s$ of solutions satisfying \eqref{eq: sums power}. 
		
Reciprocally, any solution $(x_1,\ldots,x_s) \in (\ff_{q}^*)^s$ of \eqref{eq: sums power} defines an $s$-walk from $v$ to $w$ in $\G(\kappa,q)$, by taking into account that there are $d_\kappa$ elements $y\in \ff_{q}^*$ such that 
\eqref{eq: vec x^k = vec y^k} holds 
for each $x\in \ff_{q}^*$. 
Thus, there are $(d_\kappa)^s$ different solutions of \eqref{eq: sums power} which induce the same $s$-walk. 
Therefore, we get \eqref{eq: walk sol} 
as desired.
\end{proof}

In particular, if $w=v$, we have that 
\begin{equation}
	N_{m \times s, q}^*(\kappa) = (d_\kappa)^s \, w_{\G(\kappa,q)}(s;v,v).
\end{equation} 	
That is, closed walks are related with the number of solutions of monic homogeneous systems of diagonal equations.

\begin{rem}
From Lemma \ref{lema: walks} and \eqref{eq: wG(s)}, by adding over the vertices one can obtain the total number of $s$-walks $w_\G(s)$ and of closed (resp.\@ open) walks $w_\G^{cl}(s)$ (resp.\@ $w_\G^{op}(s)$) of length $s$ in $\G=\G(\kappa,q)$:  
	$$ 
		w_\G(s) = \tfrac{1}{(d_\kappa)^s} \Big\{ \sum_{v \in (\ff_q)^m} N_{m \times s,q}^*(\kappa) +  \tfrac 12 \sum_{\substack{v, w \in (\ff_q)^m \\ v\ne w}} N_{m \times s,q}^*(\kappa, w-v) \Big\}. 
	$$
That is, $w_\G(s) = w_\G^{cl}(s) + w_\G^{op}(s)$, where 
\begin{align*}
	& w_\G^{cl}(s) = \tfrac{1}{(d_\kappa)^s} \sum_{v \in (\ff_q)^m} N_{m \times s,q}^*(\kappa), \\ 
	& w_\G^{op}(s) = \tfrac{1}{2(d_\kappa)^s} \sum_{\substack{v, w \in (\ff_q)^m \\ v\ne w}} N_{m \times s,q}^*(\kappa, w-v) .
\end{align*}
On the other hand, by adding over $s$, one can obtain the total number of walks $w_\G$, open walks $w_\G^{op}$ and closed walks $w_\G^{cl}$ of $\G=\G(\kappa, q)$.
\end{rem}

\subsubsection*{Recursive formulas for $N_{m\times s,q}(\kappa,\beta)$}
Here, we show that it is possible to recursively obtain the number $N_{m \times s, q}(\kappa, w-v)$ of solutions  
in $(\ff_q)^s$ of the system 
	$$ \mathcal{S}_{m \times s, q} (\kappa, w-v)$$ 
in \eqref{eq: monic SDE}, with $v,w \in (\ff_q)^m$,
in terms of the numbers $N_{m \times \ell, q}^*(\kappa, w-v)$ of solutions in $(\ff_q^*)^\ell$ with non-zero coordinates of smaller systems $\mathcal{S}_{m \times \ell, q} (\kappa, w-v)$ with $1\le \ell \le s$.

Thus, in the previous notations, we have that the following.

\goodbreak 
\begin{prop} \label{prop: recursion of Nmxsq}
Let $m,s \in \N$, $\kappa \in \N^m$ and $v,w\in (\ff_q)^m$. The following holds:
\begin{enumerate}[$(a)$]
	\item If $v\ne w$, then 
	\begin{equation} \label{eq: recursive fla v,w}
		 N_{m \times s,q}(\kappa, w-v) = \sum\limits_{\ell=1}^{s} \tbinom{s}{\ell} N_{m \times \ell,q}^*(\kappa, w-v).
	\end{equation}

	\item If $v = w$, then we have that
	\begin{equation} \label{eq: recursive fla v,v}
		N_{m \times s,q}(\kappa, 0) = 1+\sum\limits_{\ell=1}^{s} \tbinom{s}{\ell} N_{m \times \ell,q}^*(\kappa, 0).
	\end{equation}
\end{enumerate}
\end{prop}

\begin{proof}
($a$) Suppose first that $v\ne w$, let $\kappa=(k_1, \ldots,k_m) \in \N^{m}$, and consider a solution $(x_1,\ldots,x_s)$ of the system of equations $X_1^{k_i}+\cdots+X_s^{k_i}=w_i-v_i$ with $i=1,\ldots,m$. 
We can assume that there are $\ell$ non-zero coordinates of $(x_1,\ldots,x_s)$ for any $\ell=1,\ldots, s$. 
In this case, for any $i=1,\ldots,\ell$, there is a sequence $1\le r_1 < r_2 < \cdots < r_{\ell} \le s$ such that $x_{r_i} \ne 0$ and $x_j=0$ for any $j \ne r_{i}$.
Hence, we have that
	$$
		x_{r_1}^{k_i}+\cdots +x_{r_{\ell}}^{k_i}=x_1^{k_i}+\cdots+x_s^{k_i}=w_i-v_i \qquad (1\le i \le m).
	$$

On the other hand, given $\ell=1,\ldots,s$, any solution $(y_1,\ldots,y_{\ell})\in (\ff_{q}^*)^{\ell}$ of the system of diagonal equations $X_1^{k_i}+\cdots+X_{\ell}^{k_i}=w_i-v_i$ with $i=1,\ldots,m$, induces a solution $(x_1,\ldots,x_s) \in \ff_q^{s}$ of the system of equations $X_1^{k_i} + \cdots + X_s^{k_i} = w_i-v_i$ with $i=1,\ldots,m$ by choosing a sequence $1\le n_1 < \cdots < n_{\ell} \le s$ such that 
	$ x_{n_t}=y_t $ 
for any $t=1\ldots,\ell$, and $x_j=0$ for the remaining $j \neq n_t$. 
Since this can be achieved in $\binom{s}{\ell}$ ways, we obtain \eqref{eq: recursive fla v,w}.

\noindent ($b$) 
Now, assume that $v=w$. In this case, the formula \eqref{eq: recursive fla v,v} is obtained in the same way as \eqref{eq: recursive fla v,w}, by taking into account that we also have the trivial solution $x_i=0$ for each $i=1,\ldots,s$.
\end{proof}

Combining the previous results, we obtain the number of solutions of the SDEs in terms of walks of diagonal GP-graphs.
\begin{thm} \label{thm: Nmsq from walks}
Let $m,s \in \N$, $\kappa \in \N^m$ and $\beta\in (\ff_q)^m$. Then, the number of solutions of the monic system $\mathcal{S}_{m \times s, q} (\kappa, \beta)$ as in \eqref{eq: monic SDE} is given by 	
\begin{equation}
N_{m \times s,q}(\kappa, \beta)=
		\delta_{\beta,0} + \sum\limits_{\ell=1}^{s} \tbinom{s}{\ell} \, d_{\kappa}^{\ell} \, w_{\G}(\ell;0,\beta),  
\end{equation}
with the usual notation $\delta_{\beta,0}=1$ if $\beta=0$ and $\delta_{\beta,0}=0$ if $\beta \ne 0$. 
\end{thm}

\begin{proof}
This is straightforward from Lemma \ref{lema: walks} and Proposition \ref{prop: recursion of Nmxsq}.
\end{proof}

\begin{rem} \label{rem: binomial inversion fla}
Of course, by applying to Theorem \ref{thm: Nmsq from walks} the binomial inversion formula 
	$$ b_n = \sum_{k=0}^n \tbinom nk a_k \qquad \Rightarrow \qquad
		a_n = \sum_{k=0}^n (-1)^{n-k} \tbinom nk b_k,$$
one can obtain the number of walks in terms of the number of solutions of SDEs.
\end{rem}

\subsection{Adjacency matrices}	
Now, we put together the previous results in the section to obtain the number of solutions of the general monic (non-homogeneous) system of diagonal equations in terms of the adjacency matrix of the associated diagonal GP-graph. 

As a matter of notation, the adjacency matrix $A_{\Gamma}$ of $\Gamma=\G(\kappa,q)$ is in $M_{q^m}(\ff_q)$ since the vertex set of the graph is $V = (\ff_q)^m$. In this case, the elements of the matrix are indexed by vectors $v \in (\ff_q)^m$. If we order the elements of $V$ as $\{v_1, v_2,\ldots,v_{q^m} \}$, then $A_{v_i,v_j}$ denotes the element in the $i$-th row and $j$-th column indexed by $v_i$ and $v_j$, respectively.
Also, we will denote by $I_n$ the $n \times n$ identity matrix.

\begin{thm} \label{teo: Nmsq no-homogeneo con matrices}
Let $m,s \in \N$, $\kappa$ as in \eqref{eq: vector kappa}, and $\beta\in (\ff_q)^m$. Then, the number of solutions of the monic system $\mathcal{S}_{m \times s, q} (\kappa, \beta)$ as in \eqref{eq: monic SDE} is given by 
\begin{equation} \label{eq: nonhomog}
	N_{m \times s,q}(\kappa, \beta) = \big( (I_{q^m}+d_{\kappa}A_{\G})^{s}\big)_{0,\beta},
\end{equation}
where $d_\kappa$ is as in \eqref{eq: k hat} and $A_{\Gamma}$ is the adjacency matrix of $\Gamma=\G(\kappa,q)$.
\end{thm}

\begin{proof}
By taking into account that $w_{\G}(\ell;0,\beta)$ is the $(0,\beta)$-coordinate of the matrix $A_{\G}^{\ell}$, by Theorem \ref{thm: Nmsq from walks} we have that
	$$
		N_{m \times s,q}(\kappa, \beta) = \delta_{\beta,0} +
		\sum\limits_{\ell=1}^{s} \tbinom{s}{\ell} \, d_{\kappa}^{\ell} \, \big(A_{\G}^{\ell}\big)_{0,\beta}. 
	$$
Hence, by the addition and constant multiplication of matrices we have that 
	$$
		N_{m \times s, q}(\kappa, \beta) = \delta_{\beta,0} +
		\Big(\sum\limits_{\ell=1}^{s} \tbinom{s}{\ell} \, d_{\kappa}^{\ell} \, A_{\G}^{\ell}\Big)_{0,\beta} 
	$$
On the other hand, since $I=I_{q^m}$ and $A_{\G}$ commutes, we have that 
	$$
	(I+d_{\kappa}A_{\G})^{s} = \sum_{\ell=0}^{s}\tbinom{s}{\ell}(d_{\kappa}A_{\G})^{\ell} = I+\sum_{\ell=1}^{s}\tbinom{s}{\ell}d_{\kappa}^{\ell} A_{\G}^{\ell}.
	$$
Thus, by taking the $(0,\beta)$-coordinate in the above equality we obtain
	$$
	\big((I+d_{\kappa}A_{\G})^{s}\big)_{0,\beta} = \delta_{\beta,0} +
		\Big(\sum\limits_{\ell=1}^{s} \tbinom{s}{\ell} \, d_{\kappa}^{\ell} \, A_{\G}^{\ell}\Big)_{0,\beta} 
	$$
Therefore, we obtain \eqref{eq: nonhomog}, as asserted.
\end{proof}

\begin{rem}\label{rem: diagonalization}
Recall that the adjacency matrix of a Cayley graph is a normal matrix that commutes with its conjugate transpose. Thus, $A_\G$ is diagonalizable and hence there exists an invertible matrix $R_{\G}$ such that 
	$$ A_{\G} = R_{\G}^{-1}D_{\G}R_{\G}  \qquad \text{ with } \qquad D_{\G} = diag(\lambda_1, \ldots, \lambda_{q^{m}})$$
where $D_{\G}$ is diagonal and the $\lambda_i$'s are the eigenvalues of $A_\G$. In this case, we have
\begin{align*}
	(I_{q^m}+d_{\kappa}A_{\G})^{s}  
	& = (R_{\G}^{-1} (I_{q^m}+d_{\kappa}D_{\G}) R_{\G})^{s} = R_{\G}^{-1} (I_{q^m}+d_{\kappa}D_{\G})^{s} R_{\G} \\ 
    & = R_{\G}^{-1} \, diag\big( (1+d_{\kappa} \lambda_1)^{s},\ldots,(1+d_{\kappa} \lambda_{q^{m}})^{s} \big) \, R_{\G}.
\end{align*} 	
Therefore, we can obtain $N_{m\times s,q}(\kappa,\beta)$'s from the above matrix multiplications. 
\end{rem}

\subsubsection*{Connection with the Ihara zeta function}
We now show that the logarithmic expansion $\log \zeta_\Gamma(u)$ of the Ihara zeta function of $\Gamma=\G(\kappa,q)$ is completely determined by the number of  solutions $N_{m \times \ell, q}(\kappa)$ to the homogeneous diagonal systems.

\begin{prop} \label{prop: Ihara and solutions}
Let $m, s \in \mathbb{N}$ and $\kappa \in \mathbb{N}^m$.  Let $q$ be a prime power such that either $q$ is even, or $q$ is odd with $D=\frac{q-1}{d_{\kappa}}$ and $\frac{k_i}{d_\kappa}$ odd for all $i=1,\dots,m$. Let $\Gamma = \Gamma(\kappa, q)$ be the associated undirected diagonal GP-graph and let $\rho_\Gamma
=\frac{q^m}{2}(D-2)$.
Then, the following identity holds in $\mathbb{Q}[[u]]$:
$$
\log\zeta_\Gamma(u) = \rho_\Gamma\sum_{k=1}^{\infty}\tfrac{u^{2k}}{k}+q^m\sum_{k=1}^{\infty}\tfrac{1}{k}\sum_{j=0}^{k}
	\tbinom{k}{j}\tfrac{(D-1)^{k-j}u^{2k-j}}{d_\kappa^{\,j}}
	\sum_{\ell=0}^{j}
	(-1)^{k-\ell}\tbinom{j}{\ell}
	N_{m\times\ell,q}(\kappa),
$$
where $N_{m\times0,q}(\kappa)=1$.
In particular, the Ihara zeta function of $\Gamma$ is
completely determined by the numbers
$N_{m\times\ell,q}(\kappa)$, $\ell\geq0$.
\end{prop}

\begin{proof}
Let $A$ be the adjacency matrix of $\G$. By Proposition \ref{prop: Ihara Zeta GP graph} we have that 
$$
\zeta_\Gamma(u)^{-1}
=(1-u^2)^{\rho_\Gamma}
\det\bigl(I-uA+(D-1)u^2I\bigr).
$$
Taking formal logarithms, we obtain
\[
\log\zeta_\Gamma(u)
=\rho_\Gamma\sum_{k=1}^{\infty}\frac{u^{2k}}{k}
+\sum_{k=1}^{\infty}\tfrac{1}{k}
\operatorname{Tr}\bigl((uA-(D-1)u^2I)^k\bigr).
\]
Since $A$ commutes with $I$, the binomial theorem gives
$$
\operatorname{Tr}\bigl((uA-(D-1)u^2I)^k\bigr)
=
\sum_{j=0}^{k}
\tbinom{k}{j}(-1)^{k-j}(D-1)^{k-j}
u^{2k-j}\operatorname{Tr}(A^j).
$$
For $j\geq0$. By vertex-transitivity and Lemma \ref{lema: walks} we obtain that
$$
\operatorname{Tr}(A^j)
=q^m w_\Gamma(j;0,0)
=\tfrac{q^m}{d_\kappa^{\,j}}
N^*_{m\times j,q}(\kappa).
$$
On the other hand, Proposition \ref{prop: recursion of Nmxsq} and the convention $N^*_{m\times0,q}(\kappa)=N_{m\times0,q}(\kappa)=1$ imply that
	$$ N_{m\times j,q}(\kappa) = \sum_{\ell=0}^{j} \tbinom{j}{\ell}N^*_{m\times\ell,q}(\kappa). $$
Thus binomial inversion yields
	$$ N^*_{m\times j,q}(\kappa) = \sum_{\ell=0}^{j} (-1)^{j-\ell} \tbinom{j}{\ell} N_{m\times\ell,q}(\kappa). $$
Substituting these identities into the logarithmic
expansion and using
$(-1)^{k-j}(-1)^{j-\ell}=(-1)^{k-\ell}$
proves the asserted formula.
All substitutions are valid formally, since
$2k-j\geq k$ for $0\leq j\leq k$, so each coefficient
receives contributions from only finitely many terms.
\end{proof}


\section{Monic SDEs and the spectrum of $\G(\kappa,q)$} \label{sec: Homogeneous SDEs and spectrum}
In the previous section, we showed that the number of solutions of monic SDEs over finite fields can be put in terms of structural data of associated diagonal GP-graphs.
In this section, we show that the number of solutions over finite fields of monic homogeneous SDEs can be expressed in terms of the spectrum of the associated diagonal GP-graphs. In the case that the system is homogeneous, we can say something better.

Recall that given a graph $\G$ with adjacency matrix $A$, the spectrum of $\G$ is the set of eigenvalues of $A$ counted with multiplicities. It is denoted by 
$$Spec(\G) = \{ [\lambda_1]^{\mu_1},  \ldots, [\lambda_N]^{\mu_N} \},$$ 
where the $\lambda_i$'s are the eigenvalues and the $\mu_i$'s the multiplicities.

\subsection{The general case}
The characters of $(\ff_q)^m$ are given by 
$$ \chi_{\beta}(x_1,\ldots, x_m)=\chi_{\beta_1}(x_1)\cdots \chi_{\beta_{m}}(x_m), $$
where $\beta=(\beta_1,\ldots,\beta_m)$ runs through all $(\ff_{q})^m$,
and $\chi_{\beta_{\ell}}$ are characters of $\ff_q$, that is 
$ \chi_{\beta_{\ell}}(x)=e^{\frac{2\pi i}{p}\Tr_{q/p}(\beta_{\ell} x)}$ 
for all $\ell=1,\ldots,m$.

\goodbreak 

As we mentioned in the introduction, given $\alpha=(\alpha_1,\ldots, \alpha_m)\in (\ff_q)^{m}$, the eigenvalues of $\G(\kappa,q)$ are given by (see \eqref{eq: eigenvalues GP characters} in the proof of Theorem \ref{teo: Spec G(K,q)})
\begin{equation} \label{eq: lambda a,k,q}
	\lambda_{\alpha,\kappa,q}  = \chi_\alpha(R_\kappa) = \sum_{(x_{1},\ldots, x_{m}) \in R_{\kappa}} \chi_{\alpha}(x_{1},\ldots,x_{m}).	
\end{equation}

Suppose we index the elements of $G=(\ff_{q})^{m}$ by $g_1, g_2, \ldots, g_{q^{m}}$ in some way, with $g_1=0$. 
Since diagonal GP-graphs are reversed Cayley graphs (see the Introduction) we have that
the eigenvector associated to $\lambda_{\alpha,\kappa,q}$ is 
\begin{equation} \label{eq: eigenvector}
	v_{\chi_{\alpha}}=(\overline{\chi_{\alpha}(g_1)},\ldots, \overline{\chi_{\alpha}(g_{q^{m}})})^t.
\end{equation}
Also, by orthogonality of the characters, the vector 
\begin{equation}\label{eq: eigenvectors orthonormales}
	w_{\chi_{\alpha}} = \tfrac{1}{q^{m/2}} \, v_{\chi_{\alpha}}
\end{equation}
has complex norm $1$. 	

We are now in a position to give a general formula for the number of solutions of monic systems of diagonal equations over finite fields in terms of the spectrum of diagonal GP-graphs $\G(\kappa,q)$ and of characters of $\ff_{q^m}$.	
Recall that for monic systems we assume that $\kappa = (k_1,\ldots,k_m)$ satisfy $k_1 < \cdots < k_m$. We denote this by saying that $\kappa \in \N_<^m$.

\begin{thm} \label{teo: Nmsq chars y spec}
	Let $m,s \in \N$, $\kappa \in \N_<^m$ and $\beta\in (\ff_q)^m$. Then, the number of solutions of the monic system $\mathcal{S}_{m \times s, q} (\kappa, \beta)$ as in \eqref{eq: monic SDE} is given by 
	\begin{equation} \label{eq: gral fla}
		N_{m \times s,q}(\kappa, \beta) = \tfrac{1}{q^m} \sum_{\alpha \in \ff_{q^m}} \overline{\chi_\beta(\alpha)} \, \big( 1+d_\kappa \lambda_{\alpha, \kappa,q} \big)^s
	\end{equation}
	where $\lambda_{\alpha, \kappa,q}$ are the eigenvalues of $\G(\kappa,q)$ as given in \eqref{eq: lambda a,k,q} and $d_\kappa$ is as in \eqref{eq: k hat}.
\end{thm}

\begin{proof}
	Put $\G=\G(\kappa,q)$. 
	Let us consider the order labeling of $A_{\G}$ given by $g_1, g_2,\ldots, g_{q^{m}}$ with $g_1=0$.
	By Theorem~\ref{teo: Nmsq no-homogeneo con matrices} and Remark \ref{rem: diagonalization}, we have that
	$$
	N_{m \times s,q}(\kappa, \beta)= \big( R_{\G}^{-1} \, D_{I,\G} \, R_{\G} \big)_{0,\beta}.
	$$
	where 
	$ D_{I,\G}=diag\big( (1+d_{\kappa} \lambda_1)^{s},\ldots,(1+d_{\kappa} \lambda_{q^{m}})^{s} \big)$
	and $\lambda_i \in Spec(\G)$.
	
	Since the eigenvalues of $\G(\kappa,q)$ are given by $\lambda_{\alpha,\kappa,q}$ as in \eqref{eq: lambda a,k,q} with associated unitary eigenvector $w_{\chi_{\alpha}}$ as in \eqref{eq: eigenvectors orthonormales}, when $\alpha$ runs into $(\ff_q)^{m}$, we obtain that 
	$$
	R_{\G}= (w_{\chi_{g_1}} \cdots w_{\chi_{g_{q^{m}}}}).
	$$
	By orthogonality properties of eigenvalues, we have that $R_{\G}^{-1}=R_{\G}^*$ (transpose conjugate). 
	Notice that
	$$
	diag\big((1+d_{\kappa} \lambda_1)^{s}, \ldots, (1+d_{\kappa} \lambda_{q^{m}})^{s} \big) 
	(w_{\chi_{g_1}} \cdots  w_{\chi_{g_{q^{m}}}}) =
	(z_{\chi_{g_1}} \cdots z_{\chi_{g_{q^{m}}}})
	$$
	with $z_{\chi_{\beta}}$ the column vector given by
	$$
	{\purple	z_{\chi_{\beta}}=\tfrac{1}{q^{m/2}}\big((1+d_{\kappa} \lambda_1)^{s} \, \overline{\chi_{\beta}(g_1)},\ldots, (1+d_{\kappa} \lambda_{q^{m}})^{s} \, \overline{\chi_{\beta}(g_{q^m})}\big)^{t}.}
	$$
	
	On the other hand, by taking into account that  the first column and row of $A_{\G}$ are labeled with $0$, we have that $\chi_{g_1}=\chi_{0}$ is the trivial character and hence the first row of $R_{\G}^{*}$ is the constant row with coefficients $\frac{1}{q^{m/2}}$. Thus, a direct matrix computation allow  us to conclude that the  $(0,\beta)$-coordinate of the matrix $R_{\G}^{*} \, D_{I,\G} \, R_{\G}$
	with $\lambda_{\ell}= \lambda_{g_{\ell},\kappa,q}$ is exactly
	$$
	(\tfrac{1}{q^{m/2}},\ldots ,\tfrac{1}{q^{m/2}})\cdot  z_{\chi_{\beta}}^{t}= \tfrac{1}{q^m} \sum_{\alpha \in \ff_{q^m}}  \overline{\chi_\beta(\alpha)} \, \big( 1+d_\kappa \lambda_\alpha \big)^s.
	$$
	This conclude the proof.
\end{proof}

\subsubsection*{Weil sums} 
We now give an alternative expression for the numbers $N_{m \times s,q}(\kappa, \beta)$ entirely in terms of the canonical additive character of $\ff_q$. In particular, we use the Weil sums associated to this character.

Recall that $f(x)$ is non-degenerate if $f(x) \ne g(x)^p-g(x)$ for any $g(x) \in  \ff_q[x]$, where $p$ is the characteristic of $\ff_q$.

\begin{coro} \label{coro: weil bound}
	Let $q$ be a prime power, $s,m \in \N$, let $\kappa =(k_1,\ldots,k_m) \in \N_<^{m}$ and $\beta \in  \ff_q$.
	
	\noindent $(i)$ 
	The number of solutions of the monic system $\mathcal{S}_{m \times s, q} (\kappa, \beta)$ as in \eqref{eq: monic SDE} is given by 
	\begin{equation} \label{eq: M spec alternativo gen}
		N_{m \times s, q} (\kappa,\beta) = \tfrac{1}{q^{m}} \sum_{\alpha \in \ff_q^m} \overline{\chi(\alpha \cdot \beta)} \big( \mathcal{W}_\chi (f_{\alpha,\kappa}) \big)^{s}
	\end{equation}
	with 
	$ f_{\alpha,\kappa}(x) = \alpha_1 x^{k_1} + \cdots + \alpha_m x^{k_m} \in \ff_q[x] $ 
	and where $\chi$ is the canonical additive character of $\ff_q$ as defined in \eqref{eq: canonical character},  $\mathcal{W}_\chi(f_{\alpha,\kappa})$ stands for the associated Weil sum as in \eqref{eq: Weil sum} and $\alpha \cdot \beta$ denotes the usual inner product on $\ff_{q^m}$. 
	
	\noindent $(ii)$ 
	Furthermore, if we assume that 
	$f_{\alpha,\kappa} \in \ff_q[x]$ is non-degenerate for all $\alpha \in (\ff_q)^m$ with $\alpha \ne (0,\ldots,0)$, 
	then we have the Weil-type bound
	\begin{equation} \label{eq: weil bound gen}
		\left| N_{m \times s, q} (\kappa,\beta) - q^{s-m} \right| \le (q^m-1) \, q^{\frac s2-m} (k_m-1)^s,
	\end{equation}
	which does not depend on $\beta$.
\end{coro}

\begin{proof}
	($i$) 
	Notice that 
	$$
	\overline{\chi_{\beta}(\alpha)}= e^{-\frac{2\pi i}{p}\Tr_{q/p}(\alpha_1 \beta_1+\cdots +\alpha_{m} \beta_{m})}=\overline{\chi(\alpha \cdot \beta)},
	$$
	with $\chi$ the canonical additive character of $\ff_q$.
	Therefore, the equation \eqref{eq: M spec alternativo gen} is a direct consequence of Theorem \ref{teo: Nmsq chars y spec} and \eqref{eq: autovalores de G(kapkka,q)}.

	\noindent ($ii$) 
	We separate the contribution of the trivial character ($\alpha = 0$) in the sum in \eqref{eq: M spec alternativo gen}. For $\alpha = 0$, we have $f_{0,\kappa}(x) = 0$ and $\chi(0\cdot \beta)=1$, which yields $\mathcal{W}_\chi(0) = q$. Thus, 
	$$ N_{m\times s,q}(\kappa,\beta) = q^{s-m} + \tfrac{1}{q^m} \sum_{\alpha \ne 0} \overline{\chi(\alpha \cdot \beta)} \big( \mathcal{W}_\chi (f_{\alpha,\kappa}) \big)^s. $$
	Now we use the Weil bound: if $f(x)$ is a non-degenerate polynomial in $\ff_q[x]$ of degree $d$, then $|\mathcal{W}_\chi(f)| \le (d-1) \sqrt q.$ 
	Applying this bound to the remaining $q^m - 1$ terms (since $f_{\alpha,\kappa}$ is non-degenerate for $\alpha \ne 0$ and its degree is at most $k_m$) and taking into account that $|\chi(\alpha \cdot \beta)|=1$, we obtain
	$$	
	\left| N_{m\times s,q}(\kappa,\beta ) - q^{s-m} \right|  \le \tfrac{1}{q^m} \sum_{\alpha \ne 0} | \mathcal{W}_\chi (f_{\alpha,\kappa}) |^s 
	\le \tfrac{q^m-1}{q^m} \, q^{\frac s2} \, (k_m-1)^s, 
	$$ 
	from which \eqref{eq: weil bound gen} follows, thus concluding the proof.
\end{proof}

\subsection{The homogeneous case}
We now focus on the case of homogeneous systems of monic diagonal equations, that is the SDE in \eqref{eq: monic SDE} with 
$\beta_1=\cdots =\beta_m=0$. Namely, for $\kappa = (k_1,\ldots, k_m) \in \N_<^m$, the system  
	\begin{equation} \label{eq: homog SDE}
		\mathcal{S}_{m \times s} (\kappa) :=  \quad
		\begin{cases}
			\begin{tabular}{ccccccc}	
				$X_1^{k_1}$ & $+$ & $\cdots$ & $+$ & $X_s^{k_1}$  & $=$ & $0$, \\[1mm]
				$\vdots$          &       & $\vdots$ & & $\vdots$  &  & $\vdots$ \\[1mm]
				$X_1^{k_m}$ & $+$ & $\cdots$ & $+$ & $X_s^{k_m}$  & $=$ & $0$.   
			\end{tabular}
		\end{cases}
	\end{equation}
Since a monic homogeneous system has only coefficients $0$'s and $1$'s, we have dropped the $q$ from the notation and just write $\mathcal{S}_{m \times s} (\kappa)$ instead of $\mathcal{S}_{m \times s, q} (\kappa)$.

We now show that the number of solutions of the system $\mathcal{S}_{m \times s, q} (\kappa)$ in \eqref{eq: homog SDE} can be expressed in terms of the spectrum of the diagonal GP-graph $\G(\kappa,q)$. 
	
\begin{thm} \label{teo: M with Spec}
Let $q$ be a prime power, $s,m \in \N$ and let $\kappa =(k_1,\ldots,k_m) \in \N_<^{m}$.
The number of solutions in $(\ff_q)^s$ of the homogeneous system $\mathcal{S}_{m \times s, q} (\kappa)$ in \eqref{eq: homog SDE} is given by 
\begin{equation} \label{eq: M spec}
	N_{m \times s, q} (\kappa) = \tfrac{1}{q^{m}} \sum_{\lambda \in Spec(\G)} (1 + d_\kappa \, \lambda )^{s}
\end{equation}	
where $\G=\G(\kappa,q)$ and $d_\kappa$ is as defined in \eqref{eq: k hat}.
\end{thm}

\begin{proof}
Let $s$ and $m$ be fixed.  
Notice that, by \eqref{eq: walk sol} in Lemma \ref{lema: walks}, for any $v\in (\ff_q)^m$ and for each $1 \le \ell \le s$ we have 
	$$w_{\G(\kappa,q)}(\ell;v,v) = \frac{1}{d_\kappa^{\ell}} \, N_{m \times \ell,q}^*(\kappa,0).$$
On the other hand, it is well-known that 
	$$w_{\G(\kappa,q)}(\ell;v,v) = (A_\G^{\ell})_{v,v},$$
where $A_\G$ is the adjacency matrix of $\G=\G(\kappa,q)$.
Hence, putting these facts together we arrive at 
	$$ \Tr (A_\G^{\ell}) = \sum_{v\in \ff_{q}^{m}}(A_\G^{\ell})_{v,v} = \frac{q^{m}}{d_\kappa^{\ell}} \, N_{m \times \ell,q}^*(\kappa,0).$$
Thus, using that the trace of a matrix is the sum of its eigenvalues, and since the eigenvalues of $A_\G^{\ell}$ are  $\lambda^{\ell}$ with $\lambda$ an eigenvalue of $A_\G$, 
we obtain that
	$$ \sum_{\lambda\in Spec(\G)} \lambda^{\ell} = \frac{q^{m}}{d_\kappa^{\ell}} \, N_{m \times \ell,q}^*(\kappa,0) $$ 
for all $1 \le \ell \le s$.
By \eqref{eq: recursive fla v,w}, we have that
\begin{equation} \label{eq: binomial aux}
	\begin{aligned}
		N_{m,s} &= 1+\sum_{\ell=1}^{s} \tbinom{s}{\ell} N_{m \times \ell,q}^* (\kappa,0)\\ 
		&= 1+\tfrac{1}{q^{m}} \sum_{\ell=1}^{s} \sum_{\lambda\in Spec(\G)} \tbinom{s}{\ell} \lambda^{\ell} 			d_\kappa^{\ell} = \tfrac{1}{q^{m}} \sum_{\lambda \in Spec(\G)} \sum_{\ell=0}^{s} \tbinom{s}{\ell} \lambda^{\ell}d_\kappa^{\ell}.		
	\end{aligned}
\end{equation}
Therefore, by using the binomial theorem we obtain \eqref{eq: M spec} as asserted. 
\end{proof}
	
Notice that, if $Spec(\Gamma(\kappa,q)) = \{ [\lambda_1]^{m_1}, \ldots, [\lambda_t]^{m_t}\}$ then \eqref{eq: M spec} takes the form
	\begin{equation} \label{eq: M spec mult}
		N_{m \times s,q}(\kappa) = \tfrac{1}{q^{m}} \sum_{i=1}^t m_i \, (1+ \lambda_i d_\kappa)^{s}.
	\end{equation}

\begin{rem}
The previous theorem can also be obtained from our results from adjacency matrices as we next explain. 
However, we have preferred to give a more graph-theoretic proof using walks.
Notice that Theorem \ref{teo: M with Spec} can be obtained straight from Theorem \ref{teo: Nmsq chars y spec} by taking $\beta=0 \in (\ff_q)^m$ since $\chi_{(0,\ldots,0)} (\alpha) = \chi_0(\alpha_1) \cdots \chi_0(\alpha_m) = 1$ where $\chi_0$ is the trivial character in $\ff_q$, i.e. $\chi_0(g)=1$ for every $g \in (\ff_q)^m$.
\end{rem}

Two graphs $G,G'$ are said to be \textit{isospectral} if they have the same spectrum, i.e. $Spec(G)=Spec(G'$). The following is automatic from Theorem \ref{teo: M with Spec}.
	
\begin{coro}
If $\G(\kappa,q)$ and $\G(\kappa',q)$ are isospectral diagonal GP-graphs and $d_\kappa = d_{\kappa'}$ then $N_{m \times s,q}(\kappa) = N_{m \times s,q}(\kappa')$ for any $s\in \N$. 
\end{coro}

Note that in the previous corollary, necessarily $m\ge 2$, since for classic GP-graphs there can not exist an isospectral non-isomorphic pair $\G(k,q)$, $\G(k',q)$.

\subsubsection*{Spectral moments}
Finally, we give an alternative expression using spectral moments.
We recall that, given $k\in \N_0$, the \textit{spectral $k$-th moment} of a graph $G$ is given by 
\begin{equation} \label{eq: momentos}
	M_k(G) = \sum_{\lambda \in Spec(G)} \lambda^k.	
\end{equation}

By a direct application of Newton's binomial, the numbers $N_{m \times s,q}(\kappa)$ can be expressed in terms of the spectral moments of $\G(\kappa, q)$.
	
\begin{coro} \label{coro: M con momentos}
In the previous notations we have 
	\begin{equation} \label{eq: M con momentos}
		N_{m \times s, q}(\kappa) = 1+ \tfrac{1}{q^{m}} \sum_{2 \le j \le s} \tbinom sj \, d_\kappa^{j} \, M_j(\G(\kappa,q)). 
	\end{equation}	
\end{coro}
	
\begin{proof}
Let $\G=\G(\kappa,q)$. Applying Newton's binomial in \eqref{eq: M spec}, that is by \eqref{eq: binomial aux}, and using the definition of the spectral moments \eqref{eq: momentos}, we have 
$$	N_{m \times s, q} (\kappa) = \frac{1}{q^{m}} \sum_{j=0}^s \tbinom{s}{j} \, d_\kappa^{j} \, M_j(\G) .$$
The result follows by using that 
$M_0(\G) = \sum_{\lambda \in \mathcal{S}} 1 = |\G| = q^m$ 
and $M_1(\G) = \sum_{\lambda \in \mathcal{S}} \lambda = 0$,
where $\mathcal{S}=Spec(\G)$.
\end{proof}

We now give a baby example so that we can show how to apply our previous results.	

\begin{exam} \label{ejem: miniCao}
Let us compute the exact number of solutions of the following monic homogeneous system of $m=2$ diagonal equations with $s=3$ variables over $\mathbb{F}_{3^6}$ given by the power vector $\kappa = (1, 4)$
\begin{equation} \label{eq: miniCao}
	\begin{cases}
		X_1 + X_2 + X_3 = 0, \\[1.25mm]
		X_1^4 + X_2^4 + X_3^4 = 0.
	\end{cases}
\end{equation}
Here, the field size is $q = 3^6 = 729$ and $d_\kappa = \gcd(1, 4, q-1) = 1$. We will solve this using three different approaches developed in Sections 3 and 4, using the undirected connected (see Proposition \ref{prop: diag GPs}) diagonal GP-graph $\G=\G(\kappa, q) = \G((1,4),729)$.
	
\noindent \textit{Method 1: using the adjacency matrix}. 
By Theorem \ref{teo: Nmsq no-homogeneo con matrices}, the number of solutions is given by the $(0,0)$-entry of the matrix $(I_{q^m} + d_\kappa A_\Gamma)^s$, where $A_\Gamma$ is the adjacency matrix of the diagonal GP-graph $\Gamma$. Since $s=3$ and $d_\kappa = 1$, we expand the binomial:
	\[ N_{2 \times 3, q}(\kappa) = \left[ (I + A_\Gamma)^3 \right]_{0,0} = [I]_{0,0} + 3[A_\Gamma]_{0,0} + 3[A_\Gamma^2]_{0,0} + [A_\Gamma^3]_{0,0}. \]
We evaluate each term by counting closed walks in $\Gamma$:
\begin{enumerate}[$\bullet$]
	\item $[I]_{0,0} = 1$, accounting for the trivial walk of length 0. \sk 

	\item $[A_\Gamma]_{0,0} = 0$, since $0 \notin R_\kappa$, meaning the graph has no loops. \sk 

	\item $[A_\Gamma^2]_{0,0} = 0$, which counts closed walks of length 2. This requires $x_1, x_2 \in \mathbb{F}_q^*$ such that $x_1+x_2=0$ (so $x_2 = -x_1$) and $x_1^4+x_2^4=0$. In characteristic 3, this implies $2x_1^4 = 0$, forcing $x_1=0$, a contradiction since $x_i \neq 0$. \sk 

	
	\item $[A_\Gamma^3]_{0,0} = 728$, which counts strictly non-zero solutions. Algebraically, substituting $x_3 = -x_1 - x_2$ into the second equation yields $x_1^4 + x_2^4 + (-x_1 - x_2)^4 = 0$. In characteristic 3, expanding this gives $x_1^4 + x_2^4 + x_1^4 + x_1^3 x_2 + x_1 x_2^3 + x_2^4 = 0$, which simplifies to 
	$$ (x_1^3 - x_2^3)(x_1 - x_2) = (x_1 - x_2)^4 = 0.$$ 
	This forces $x_1 = x_2$, and consequently $x_3 = x_1$, meaning the solutions are restricted to the diagonal $x_1 = x_2 = x_3 = x$. There are exactly $q - 1 = 728$ such non-zero elements $x \in \mathbb{F}_q^*$.
	\end{enumerate}
Summing these contributions, we obtain 
	$$ N_{2 \times 3, q}(\kappa) = 1 + 0 + 0 + 728 = 729.$$

\noindent \textit{Method 2: using the spectrum and Weil sums}. 
Alternatively, using the spectral approach for homogeneous systems, Corollary \ref{coro: weil bound} gives the number of solutions in terms of Weil sums:
	\[ N_{2 \times 3, q}(\kappa) = \tfrac{1}{q^2} \sum_{\alpha \in \mathbb{F}_q^2} \Big( \mathcal{W}_\chi(f_{\alpha,\kappa}) \Big)^3 = \frac{1}{q^2} \sum_{\alpha_1, \alpha_2 \in \mathbb{F}_q} \Big( \sum_{x \in \mathbb{F}_q} \chi(\alpha_1 x + \alpha_2 x^4) \Big)^3. \]
	Expanding the cube yields a triple sum over $x_1, x_2, x_3 \in \mathbb{F}_q$. Reordering the summation to evaluate the characters first, we get:
	\[ N_{2 \times 3, q}(\kappa) = \tfrac{1}{q^2} \sum_{x_1, x_2, x_3 \in \mathbb{F}_q} \Big( \sum_{\alpha_1 \in \mathbb{F}_q} \chi(\alpha_1(x_1+x_2+x_3)) \Big) \Big( \sum_{\alpha_2 \in \mathbb{F}_q} \chi(\alpha_2(x_1^4+x_2^4+x_3^4)) \Big). \]
	By the orthogonality of characters, the inner sums evaluate to $q$ if the respective equations are satisfied, and $0$ otherwise. Thus, each tuple $(x_1, x_2, x_3)$ that is a solution to the system contributes exactly $q \times q = q^2$ to the sum. Since the only solutions are of the form $x_1=x_2=x_3=x$ for any $x \in \mathbb{F}_q$, there are exactly $q$ valid tuples. Therefore, the sum collapses to:
	\[ N_{2 \times 3, q}(\kappa) = \tfrac{1}{q^2} (q \cdot q^2) = q = 729. \]

\noindent \textit{Method 3: using spectral moments}. 
Finally, we can compute the number of solutions via the spectral moments of the diagonal GP-graph $\Gamma = \Gamma(\kappa, q)$ as stated in Corollary \ref{coro: M con momentos}:
	\[ N_{2 \times 3, q}(\kappa) = 1 + \tfrac{1}{q^m} \sum_{2 \le j \le s} \tbinom{s}{j} \, d_\kappa^j \, M_j(\Gamma), \]
where $s=3$, $m=2$, $d_\kappa = 1$, and $M_j(\Gamma)$ denotes the $j$-th spectral moment of the graph. Substituting these values yields:
	\[ N_{2 \times 3, q}(\kappa) = 1 + \tfrac{1}{q^2} \left\{ \tbinom{3}{2} M_2(\Gamma) + \tbinom{3}{3} M_3(\Gamma) \right\} = 1 + \tfrac{3 M_2(\Gamma) + M_3(\Gamma)}{q^2}. \]
By definition, the spectral moments correspond to the total number of closed walks of length $j$ in the graph, which are related to the total non-zero solutions of the homogeneous systems:
\begin{itemize}
	\item $M_2(\Gamma) = \operatorname{Tr}(A_\Gamma^2) = 0$, since there are no closed walks of length 2 in this graph. \msk 
	
	\item $M_3(\Gamma) = \operatorname{Tr}(A_\Gamma^3) = q^m N^*_{2 \times 3, q}(\kappa) = q^2 (q - 1) = q^2(q - 1)$.
\end{itemize}
Substituting $M_2(\Gamma)$ and $M_3(\Gamma)$ back into the formula gives:
	\[ N_{2 \times 3, q}(\kappa) = 1 + \tfrac{0 + q^2(q - 1)}{q^2} = 1 + (q - 1) = q = 729. \]

As we can joyfully see, all three methods independently confirm the exact count of 729 solutions. 
Furthermore, since $N_{2 \times 3, q}(\kappa)=3^6=q=q^{s-m}$, the number of solutions is exactly the expected one by the Weil-type bound in Corollary \ref{coro: weil bound}.

Notice that $N_{2 \times 3,q}(0)=q=q^{3-2}$, that is the number of solutions is exactly the expected value, so the LHS of the Weil bound \eqref{eq: weil bound gen} in Corollary \ref{coro: weil bound} is 0.
\hfill $\diamond$
\end{exam}

\subsubsection*{The Cao-Chou-Gu's systems of two equations}
We close the section with some comments on certain monic homogeneous systems of 2 diagonal equations.
In 2016, Cao, Chou and Gu studied in Section 3 of \cite{CCG} certain monic systems of two diagonal equations. 
In particular, in Theorem~3.3, they give the number of solutions of the monic homogeneous systems of diagonal equations 
\begin{equation} \label{eq: Cao}
	\begin{cases}
		\begin{tabular}{ccccccc}	
			$X_1$        & $+$ & $\cdots$ & $+$ & $X_s$  		& $=$ & $0$, \\[1mm]
			$X_1^{p^{\ell}+1}$    & $+$ & $\cdots$ & $+$ & $X_s^{p^{\ell}+1}$  	& $=$ & $0$, 
		\end{tabular}
	\end{cases}
\end{equation}
where $s\ge 2$, $\ell\ge 1$ and $q=p^{2\ell t}$ for $p$ prime and $t \in \N$. They consider four cases: they distinguish the cases $p\mid s$ or not, and for each case, the situations $p,t$ both odd or not. So, in our notation they found 
$N_{2\times s, q} (\kappa,0)$ 
where $q=p^{2\ell t}$ and $\kappa=(1,p^\ell+1)$, for any $\ell, t \in \N$, $s\ge 2$ and $p$ prime.

\begin{rem}[\textit{The Cao-Chou-Gu's systems}] \

\noindent ($i$) 
By Remark \ref{rem: binomial inversion fla}, using the expressions in Theorem 3.3 in \cite{CCG} for the number of solutions of the systems \eqref{eq: Cao}, we can obtain the number of walks in any diagonal GP-graph $\G((1,p^\ell+1),p^{2\ell t})$ for any $\ell, t\in \N$ and $p$ prime.

\noindent ($ii$) 
It is worth pointing out a minor algebraic typo in Theorem 3.3 of \cite{CCG} regarding the above systems of diagonal equations. For the cases where both $p$ and $s$ are odd 
the published formulas in Theorem 3.3 of \cite{CCG} yields fractional results for certain parameters due to uncancelled terms in the numerator during the character sum expansion. 
For the case where both $p$ and $s$ are odd, the corrected number of solutions can be expressed in a unified piecewise formula depending on the divisibility of $n$ by $p$, in our notation we have:
\begin{equation} \label{eq: cao16}
	N_{2\times s, q} (\kappa, 0) = p^{2(s-2)\ell t} + \frac{p^{(s-4)\ell t}(p^{2\ell t} - 1)}{p^\ell + 1}
	\begin{cases} 
		p^{\ell(s+2t-2)} + (-1)^s p^{\ell(2t+1)}   & \text{if } p \mid s, \\[1ex]
		p^{\ell(s+t-1)} + (-1)^{s+1} p^{\ell(t+1)} & \text{if } p \nmid s.
	\end{cases}
\end{equation}
		
\noindent ($iii$) 
The system \eqref{eq: miniCao} in Example \ref{ejem: miniCao} is a particular case of \eqref{eq: Cao}. Using this corrected formula \eqref{eq: cao16} for our previous example with $p=s=t=3$ and $\ell=1$, the numerator evaluates to $3^{3+6-2} + (-1)^3 3^{6+1} = 3^7 - 3^7 = 0$. Consequently, \eqref{eq: Cao} correctly yields
		\[ N_{2\times 3, 3^6} ((1,4),0) = 3^{2(3-2)3} + 0 = 3^6 = 729, \]
perfectly matching the exact solution count obtained via our graph-theoretic framework.
\end{rem}

\section{Hermitian-form SDEs} \label{sec: Hermitian}
Up to now we study arbitrary monic (homogeneous) systems of diagonal equations.
Now, in this and the next section we deal with monic homogeneous systems $\mathcal{S}_{m\times s, q}(\kappa)$ of a particular kind, one in which the powers $k_i$ are all of the form $q^{\ell_i}+1$ for different positive integers $\ell_i$ for $i=1,\ldots,m$. 
Namely, for $\ell=(\ell_1,\ldots,\ell_m) \in \N^m$, consider the system  
\begin{equation} \label{eq: system diag herm}
	\mathcal{S}_{m\times s, q}(\kappa_{\ell}) := \quad 
	\begin{cases}
		\begin{tabular}{ccccccc}	
			$X_1^{q^{\ell_1}+1}$ & $+$ & $\cdots$ & $+$ & $X_s^{q^{\ell_1}+1}$  & $=$ & $0$, \\[1mm]
			$\vdots$          &       & $\vdots$ & & $\vdots$  &  & $\vdots$ \\[1mm]
			$X_1^{q^{\ell_m}+1}$ & $+$ & $\cdots$ & $+$ & $X_s^{q^{\ell_m}+1}$  & $=$ & $0$.   
		\end{tabular}
	\end{cases}
\end{equation}
For consecutive odd numbers $\ell_1, \ldots, \ell_m$ we can explicitly give the number of solutions by using the known spectrum of Hermitian-form graphs.

\begin{rem}
	Notice that, although similar, the systems of two equations considered by Cao-Chou-Gu in \cite{CCG}, see \eqref{eq: Cao}, are not of the Hermitian-form type.
\end{rem}

\subsubsection*{Hermitian-form graphs}	
The \textit{Hermitian-form graph} $\mathcal{H}={\rm Her}(n,q^2)$ is the graph with vertex set the $n \times n$ Hermitian matrices with entries  in $\ff_{q^2}$, 
i.e.\@  matrices $H$ such that 
$H_{ij} = (H_{ji})^q$ 
for  all $i$ and $j$, and where two matrices $A,B$ are adjacent 
in ${\rm Her}(n,q^2)$ if and only if their difference has rank~$1$, i.e.\@ ${\rm rank}(A-B)=1$.
Hermitian-form graphs have a strong combinatorial flavor, despite their algebraic definition. 
It is well-known that Hermitian-form graphs are distance regular graphs (DRG). In fact, Her$(n,q^2)$ is a DRG with classic parameters $drg(d,b,\alpha,\beta)$ given by (see Table~6.1 in \cite{BCN})
$$ (n,b,\alpha,\beta) = (n, -q,-q-1, -(-q)^n-1).$$
Notably, Stanton showed that some affine $q$-Krawtchouk polynomials appear naturally as the spherical functions on the automorphism group of the Hermitian matrices over finite fields $\ff_{q^2}$ (\cite{St}).

Equivalently, the graph ${\rm Her}(n,q^2)$ can be seen as the Cayley graph $\Cay(\mathcal{M}_H, \mathcal{D})$, 
where $\mathcal{M}_H = M_n^{Her}(\ff_{q^2})$ is the abelian group of $n \times n$ Hermitian matrices over $\ff_{q^2}$ under the addition operation and $\mathcal{D} = \{ H \in \mathcal{H}: \mathrm{rank}(H)=1\}$.

We will take advantage of the fact that Hermitian-form graphs can also be seen as diagonal GP-graphs. In fact, 
let $n=2m$ be an even integer. Taking $\ell_i=2i-1$ for $i=1,\ldots,m$, we have  
$$ \kappa = (q^{\ell_1}+1,\ldots,q^{\ell_m}+1) = (q+1,q^3+1,q^5+1,\ldots,q^{2m-1}+1).$$
Consider the diagonal Cayley graph 
	$$ \G = \G(\kappa,q^{2n}) = \Cay(\ff_{q^{2n}}^m, R_{\kappa}) \quad \text{with} \quad 
	R_{\kappa} = \{(x^{q^{\ell_1}+1},\ldots,x^{q^{\ell_m}+1}) : x\in \ff_{q^{2n}}^*\},$$ 
that is
$R_{\kappa} = \{(x^{q+1},x^{q^3+1}, \ldots, x^{q^{2m-1}+1}) : x\in \ff_{q^{2n}}^*\}$.	

By Lemma 4.1 in \cite{ZZDX2}, we have that under these hypothesis the Cayley graph $\G$ is isomorphic to the Hermitian-form graph $\mathcal{H}$, that is 
\begin{equation} \label{eq: Her = Cay}
	{\rm Her}(n,q^2) \simeq \G(\kappa,q^{2n}).	
\end{equation}


\subsubsection*{The Hermitian-form SDEs}	
We now define a particular case of the SDEs in \eqref{eq: system diag herm}.
\begin{defi}
	For $\ell=(1,3,5,\ldots,2m-1)$, the \textit{Hermitian-form system of diagonal equations} is the system $\mathcal{S}_{m\times s, q}(A,\kappa_{\ell},\beta)$ 	defined by the equations
	\begin{equation} \label{eq: Herm form diag Eqn}
		\quad   \begin{cases}
			\begin{tabular}{ccccccc}	
				$\alpha_{11} X_1^{q+1}$ & $+$ & $\cdots$ & $+$ & $\alpha_{1s} X_s^{q+1}$  & $=$ & $\beta_1$, \\[1mm]
				$\vdots$          &       & $\vdots$ & & $\vdots$  &  & $\vdots$ \\[1mm]
				$\alpha_{m1} X_1^{q^{2m-1}+1}$ & $+$ & $\cdots$ & $+$ & $\alpha_{ms} X_s^{q^{2m-1}+1}$  & $=$ & $\beta_m$,   
			\end{tabular}
		\end{cases}
	\end{equation}
	where $A=(\alpha_{ij}) \in M_{m\times s}(\ff_{q^m})$  and $\beta=(\beta_1, \ldots, \beta_m) \in (\ff_q)^m$. 
	When $\alpha_{ij}=1$ for every $1\le i,j \le m$ (i.e., $A=J$), we say that it is a monic Hermitian-form SDE and when $\beta=0$ we say that it is an homogeneous Hermitian-form SDE.
\end{defi}

Consider the system $\mathcal{S}_{m\times s, q}(\kappa_{\ell})$ of monic homogeneous Hermitian-form diagonal equations 
\begin{equation} \label{eq: monic homogeneous Herm form diag Eqn}
	\begin{cases}
		\begin{tabular}{ccccccc}	
			$X_1^{q+1}$        & $+$ & $\cdots$ & $+$ & $X_s^{q+1}$  		& $=$ & $0$, \\[1mm]
			$X_1^{q^{3}+1}$    & $+$ & $\cdots$ & $+$ & $X_s^{q^{3}+1}$  	& $=$ & $0$, \\[1mm]
			$\vdots$           &     & $\vdots$ &     & $\vdots$  			&  & $\vdots$ \\[1mm]
			$X_1^{q^{2m-1}+1}$ & $+$ & $\cdots$ & $+$ & $X_s^{q^{2m-1}+1}$  & $=$ & $0$.   
		\end{tabular}
	\end{cases}
\end{equation}
As a direct consequence of Theorem \ref{teo: M with Spec}, we next obtain a closed formula for the number of solutions of this system in $(\ff_{q^{2n}})^s$ with $n=2m$.

\begin{thm} \label{teo: hermitiano diagonal}
	Let $q$ be a prime power, $s\in \N$, and $n=2m$ an even positive integer.
	The number of solutions in $(\ff_{q^{2n}})^s$ of the monic homogeneous Hermitian-form system $\mathcal{S}_{m\times s, q}(\kappa_{\ell})$ of $m$ diagonal equations 
	\eqref{eq: monic homogeneous Herm form diag Eqn} 
	is given by
	\begin{equation} \label{eq: Nsm}
		N_{m \times s,q^{2n}}(\kappa_{\ell}) = \frac{1}{q^{n(n-2s)}} \sum_{j=0}^{2m}  (-q)^{-js} \, Q_j(n,q)
	\end{equation}
	with $Q_0(n,q)=1$ and, for $1\le j \le n$,
	\begin{equation} \label{eq: Qj(n)}
		Q_j(n,q) = \prod_{i=0}^{j-1} \frac{q^{2n}-q^{2i}}{q^{j}-(-1)^{i+j}q^{i}} = q^{\frac{j(j-1)}{2}} \prod_{k=1}^j \frac{q^{2(n-k+1)}-1}{q^k - (-1)^k}  \in \N.
	\end{equation}	
\end{thm} 

\begin{proof}
	Let $n=2m$. Given $\kappa=(q+1,q^3+1,q^5+1,\ldots,q^{2m-1}+1)$, by \eqref{eq: Her = Cay}
	we know that 
	$\G(\kappa,q^{2n}) \simeq {\rm Her}(n,q^2)$ and hence 
	$ Spec (\G(\kappa,q^{2n})) = Spec ({\rm Her}(n,q^2))$.
	
	Now, notice that $ d_{\kappa} = \gcd(\kappa,q^{2n}-1) = q+1$.
	In the notation of Theorem~\ref{teo: M with Spec} we have that $q$ is actually $q^{2n}$, 
and hence Theorem \ref{teo: M with Spec} implies that
\begin{equation} \label{eq: M Her}
	N_{m \times s,q^{2n}}(\kappa_{\ell}) = \tfrac{1}{(q^{2n})^m} \sum_{\lambda_\G \in Spec (\G)} (1 + \lambda_\G d_{\kappa})^{s} =  \tfrac{1}{q^{n^2}} \sum_{\lambda \in Spec(\mathcal{H})} (1 + (q+1) \lambda)^{s},
\end{equation}
where $\G=\G(\kappa,q^{2n})$ and $\mathcal{H}={\rm Her}(n,q^2)$.	

The eigenvalues of ${\rm Her}(n,q^2)$ are known, they were first obtained by Stanton in \cite{St}. We give the following clearer expressions (see Corollary~8.4.4 in \cite{BCN})
\begin{equation} \label{eq: eigen Her}
	\lambda_0 = \frac{q^{2n}-1}{q+1}  \qquad \text{and} \qquad \lambda_{j} = \frac{(-q)^{2n-j}-1}{q+1}
	\quad (1\le j \le n),
\end{equation}
with multiplicities $m(\lambda_i)=Q_i(n,q)$ given by $Q_0(n,q)=1$ and 
\begin{equation} \label{eq: mult Her}
	Q_j(n,q)= {n \brack j}_{(-q)} \cdot 
	\prod_{i=0}^{j-1} \big((-1)^{n+1}q^{n}+(-1)^{i+1}q^i \big) \qquad (1\le j \le n).
\end{equation}	
where
\begin{equation} \label{eq: Gauss coefficients}
	{n \brack j}_{\ell} = \prod_{i=0}^{j-1} \frac{\ell^n-\ell^i}{\ell^j-\ell^i},
\end{equation}	
denotes the Gaussian binomial coefficients with basis $\ell \in \Z$ with $\ell\neq 1$. 

In this way, putting \eqref{eq: eigen Her} and \eqref{eq: mult Her} in \eqref{eq: M Her}, and using \eqref{eq: Gauss coefficients},  we have that
\begin{align*} 
	N_{m \times s,q^{2n}}(\kappa_{\ell}) & = \tfrac{1}{q^{n^2}} \Big((1+\lambda_0 (q+1))^s + \sum_{j=1}^n Q_j(n,q) \cdot \big( 1+\lambda_j(q+1) \big)^s \Big) \\
	& = \tfrac{1}{q^{n^2}} \Big( q^{2ns} + \sum_{j=1}^n (-q)^{s(2n-j)} \, Q_j(n,q) \Big)  = q^{-n(n-2s)} \sum_{j=0}^{n} (-q)^{-js}\, Q_j(n,q), 
\end{align*}
as asserted.

Finally, we get a closed formula for the numbers $Q_j(n,q)$'s. 
By \eqref{eq: mult Her} and \eqref{eq: Gauss coefficients} and using difference of squares, we have that 
$$	Q_j(n,q) = \prod_{i=0}^{j-1}\frac{(-q)^{n}-(-q)^{i}}{(-q)^j-(-q)^{i}} \Big((-1)^{n+1}q^{n}+(-1)^{i+1}q^i \Big) = \prod_{i=0}^{j-1}(-1)^{j+1}\frac{q^{2n}-q^{2i}}{q^{j}-(-1)^{i+j}q^i}.$$
By taking into account that $\prod_{i=0}^{j-1}(-1)^{j+1}=(-1)^{j(j+1)}=1$, we readily obtain the first equality in \eqref{eq: Qj(n)}. From here one can obtain the second equivalent expression, and the proof is complete.
\end{proof}

Using the Weil-type bound for systems that we found before, we get the following Weil-type bound for the number of solutions of homogeneous Hermitian-form SDEs.

\begin{coro}
In the previous notations, we have:
\begin{equation} \label{eq: bound her}
	\left| N_{m\times s,q^{4m}}(\kappa_\ell) - q^{4m(s-m)} \right| \le (q^{4m^2} - 1) q^{4m(s-m) - s}.
\end{equation}	
\end{coro}

\begin{proof}
Just apply the Weil-type bound \eqref{eq: weil bound gen} from Corollary \ref{coro: weil bound}. In this case, the number of equations is $m$, the finite field is $\ff_{Q}$ with $Q = q^{4m}$, and the vector of powers is $\kappa_\ell = (q+1,q^3+1,\ldots, q^{2m-1}+1)$. Hence, the maximum degree is $k_m = q^{2m-1}+1$. Substituting these values into the general bound yields
$$ \left| N_{m\times s,q^{4m}}(\kappa_\ell) - Q^{s-m} \right| \le \tfrac{Q^m-1}{Q^m} \, Q^{\frac s2} (k_m-1)^s.$$
Since $Q^{\frac s2} (k_m-1)^s = q^{2ms} (q^{2m-1})^s = q^{4ms-s}$ and $Q^{s-m} = q^{4m(s-m)}$ we get 
$\frac{q^{4m^2}-1}{q^{4m^2}} \, q^{4ms-s}$, from which the result directly follows.
\end{proof}

We now make some observations directly from the previous result on the numbers $Q_j(n,q)$, which are positive integers being multiplicities of eigenvalues.

\begin{rem}[\textit{The $Q_j$'s and the variables $X_i$'s}]
Notice that the numbers $Q_j$ in expression \eqref{eq: Qj(n)} of the previous theorem do not depend on $s$, only depend on $q$ and $m$. Therefore, once computed $Q_1,\ldots,Q_{2m}$, they can be used to compute the number of solutions for a fixed number $m$ of equations over a fixed field $\ff_{q^{2n}}$ for any number of variables $X_1,\ldots,X_s$.	
\end{rem}

\begin{rem}[\textit{The $Q_j$'s and the case $s=1$}] 
If $s=1$, the system \eqref{eq: Herm form diag Eqn} is just
$$ X^{q+1}=0, \qquad X^{q^3+1}=0, \qquad \ldots, \qquad X^{q^{2m-1}+1}=0, $$ 
which has only the trivial solution and hence $N_{m,1}=1$. 
However, using expression \eqref{eq: Nsm} with $s=1$ we get the following identity for the numbers $Q_i$: 
\begin{equation} \label{eq: identidad para los Qj's}
	\sum_{j=0}^n \tfrac{(-1)^j}{q^j} \, Q_j(n,q) = q^{n(n-2)}.
\end{equation}
Of course, from this one can obtain an expression for $Q_n$ in terms of the previous numbers $Q_1, \ldots, Q_{n-1}$. 
\end{rem}

Since we know the spectrum of the Hermitian-form graph we can explicitly compute the Ihara zeta function for the Hermitian-form graphs. 

\begin{coro} \label{cor: Ihara Hermitian}
Let $q$ be a prime power and $n=2m$ an even positive integer. 
Put $D=\frac{q^{2n}-1}{q+1}$ and $\rho_\mathcal{H}=\frac{q^{n^2}}{2}(D-2)$.
The Ihara zeta function of the Hermitian-form graph $\mathcal{H} = \operatorname{Her}(n, q^2) \simeq \Gamma(\kappa, q^{2n})$ 
with $\kappa = (q+1, q^3+1, \dots, q^{2m-1}+1)$ is given by
\begin{equation} \label{eq: Ihara Hermitian}
\zeta_\mathcal{H}(u)=\frac{(1-u^2)^{-\rho_{\mathcal{H}}}	(1-Du+(D-1)u^2)^{-1}}
{\displaystyle
	\prod_{j=1}^{n}
	\left(1-\tfrac{(-q)^{2n-j}-1}{q+1}u+(D-1)u^2\right)^{Q_j(n,q)}},
\end{equation}
where the numbers $Q_j(n,q)$ are given in \eqref{eq: Qj(n)}.
\end{coro}

\begin{proof}
We know that $\operatorname{Her}(n, q^2) \simeq \Gamma(\kappa, q^{2n})$. The graph has $|V| = q^{n^2}$ vertices 
and is regular of degree $D = \lambda_0 = \frac{q^{2n}-1}{q+1}$ . 
The number of edges is $|E| = \frac{1}{2} |V| D = \frac{1}{2} q^{n^2} \lambda_0$. Hence, 
	$$ \rho_\mathcal{H} = |E| - |V| = \tfrac{q^{n^2}}{2}(D-2),$$ 
which yields the formula for $\rho_\mathcal{H}$ stated above.
	
To evaluate the determinant in the Ihara-Bass formula \eqref{eq: Bass formula}, we use the known spectrum of $\mathcal{H}$. As shown in the proof of Theorem~\ref{teo: hermitiano diagonal}, the spectrum consists of the principal eigenvalue $\lambda_0$ with multiplicity $Q_0(n,q) = 1$, and the eigenvalues $\lambda_j = \frac{(-q)^{2n-j}-1}{q+1}$ for $1 \le j \le n$ with multiplicities $Q_j(n,q)$ (see \eqref{eq: Qj(n)}). Grouping the terms of the characteristic polynomial over the spectrum of $H$, we obtain
	\[
	\det \big( I - u A_H + u^2(\lambda_0-1)I \big) = (1 - u \lambda_0 + u^2(\lambda_0-1)) \prod_{j=1}^{n} \left( 1 - u \lambda_j + u^2(\lambda_0-1) \right)^{Q_j(n,q)}.
	\]
	Substituting this factorized determinant and $\rho_H$ into the general Ihara-Bass formula completes the proof.
\end{proof}

\section{Small Hermitian-form SDEs} \label{sec: Small HF-SDE}
In this section, we will study the Hermitian-form diagonal equation ($m=1$) and Hermitian-form SDEs with 2 and 3 equations ($m=2,3$) in more detail, obtaining closed formulas for the number of solutions.

To this end, it will be useful to have better expressions for the numbers $Q_j$'s.  
In fact, notice that $Q_j$ only depend on $q, n$ and satisfy the following recursion. 

\begin{lem} \label{lem: recursion Mjs}
	In the previous notations, for any $j\in \N$ we have that
	\begin{equation} \label{eq: recursion for Mjs}
		Q_{j+1}(n,q) = q^{j}  \frac{q^{2(n-j)}-1}{q^{j+1}+(-1)^j} \, Q_j(n,q).
	\end{equation}
	In particular, the numbers $Q_j(n,q)$ are monic polynomials in $\Z[q]$. 
\end{lem}

\begin{proof}
	Since $j\ge 1$, by \eqref{eq: Qj(n)} we have that
	$$
	Q_j(n,q) = \prod_{i=0}^{j-1} \frac{q^{2n}-q^{2i}}{q^{j}-(-1)^{i+j}q^{i}} \qquad  \text{and} \qquad 	Q_{j+1}(n,q) = \prod_{i=0}^{j} \frac{q^{2n}-q^{2i}}{q^{j+1}-(-1)^{i+j+1}q^{i}}.
	$$
	Hence, separating the first denominator and the last numerator, we have 
	\begin{align*}
		Q_{j+1}(n,q) & = \frac{q^{2n}-q^{2j}}{q^{j+1}-(-1)^{j+1}} \cdot \prod_{i=0}^{j-1} (q^{2n}-q^{2i}) \cdot \prod_{i=1}^{j} \frac{1}{q(q^{j}-(-1)^{i+j+1}q^{i-1})} \\
		& = \frac{q^{2j}(q^{2(n-j)}-1)}{q^{j}(q^{j+1}+(-1)^{j})}  \cdot  \prod_{i=0}^{j-1} (q^{2n}-q^{2i}) \cdot \prod_{i=0}^{j-1} \frac{1}{q^{j}-(-1)^{i+j}q^{i}}, 
	\end{align*}
	and hence we get \eqref{eq: recursion for Mjs}, as asserted.
	
	Finally, since $Q_j(n,q) \in \Z$ for any $j$, the remaining assertion also follows directly from \eqref{eq: recursion for Mjs} and induction.
\end{proof}

For instance, by definition we have that
	$$ Q_1(n,q) = \tfrac{q^{2n}-1}{q+1} = \tfrac{(q^n-1)(q^n+1)}{q+1} 
	= (q^n-1) (q^{n-1} - q^{n-2} + \cdots - q +1) \in \Z[q], $$
since $n$ is even. Thus, the equation \eqref{eq: recursion for Mjs} we have 
$$ Q_2(n,q) = q \, \tfrac{q^{2(n-1)}-1}{q^{2}-1} \, Q_1(n) = q \, \Psi_{n-1}(q^2) \, Q_1(n) \in \Z[q],$$
where we have used the notation $\Psi_n(x) = \frac{x^n-1}{x-1}= x^{n-1}+\cdots+x+1 \in \Z[x]$.

In the next 3 subsections, we will need the numbers $Q_1(n,q),\ldots, Q_j(n,q)$ for the values $n=2,4,6$. Here we give the general expressions for these numbers, which the reader can check by its own. 

\begin{lem} \label{lem: Q1,..,Q6}
	The numbers $Q_1(n,q),\ldots, Q_6(n,q)$ are given by 
	\begin{align*}
		& Q_1(n,q) = \frac{q^{2n}-1}{q+1}, \\[2mm]
		& Q_2(n,q) = q \frac{(q^{2n}-1)(q^{2n-2}-1)}{(q^2-1)(q+1)}, \\[2mm]
		& Q_3(n,q) = q^3 \frac{(q^{2n}-1)(q^{2n-2}-1)(q^{2n-4}-1)}{(q^3+1)(q^2-1)(q+1)}, \\[2mm]	
		& Q_4(n,q) = q^6 \frac{(q^{2n}-1)(q^{2n-2}-1)(q^{2n-4}-1)(q^{2n-6}-1)}{(q^4-1)(q^3+1)(q^2-1)(q+1)}, \\[2mm]
		& Q_5(n,q) = q^{10} \frac{(q^{2n}-1)(q^{2n-2}-1)(q^{2n-4}-1)(q^{2n-6}-1)(q^{2n-8}-1)}{(q^5+1)(q^4-1)(q^3+1)(q^2-1)(q+1)}, \\[2mm]
		& Q_6(n,q) = q^{15} \frac{(q^{2n}-1)(q^{2n-2}-1)(q^{2n-4}-1)(q^{2n-6}-1)(q^{2n-8}-1)(q^{2n-10}-1)}{(q^6-1)(q^5+1)(q^4-1)(q^3+1)(q^2-1)(q+1)}.
	\end{align*}
\end{lem}

\subsection{The Hermitian-form diagonal equation}
We now consider $m=1$ and give the number of solutions of the monic homogeneous Hermitian-form diagonal equation 
\begin{equation} \label{eq: m=1}
	X_1^{q+1}+\cdots+ X_s^{q+1} =  0.
\end{equation}

It is possible to give a closed formula for the solutions in $\ff_{q^4}$ of the equation \eqref{eq: m=1}.

\begin{prop} \label{prop: N_1xs}
	Let $q$ be a prime power and $s\in \N$.
	The number of solutions in $(\ff_{q^4})^s$ of the Hermitian-form diagonal equation $X_1^{q+1} + \cdots + X_s^{q+1}=0$ 
	is given by
	\begin{equation} \label{eq: N1s}
		N_{1 \times s,q^4} (q+1) = q^{2s-3} \big\{ q^{2s-1} + (q-1) (q^2+1) \big((-1)^{s}q^{s-1} +1 \big)  \big\}.
	\end{equation}
	In particular, $N_{1 \times s,q^4} (q+1) \equiv 0 \pmod q$ for every $s\ge 2$.
\end{prop} 	

\begin{proof}
Applying Theorem \ref{teo: hermitiano diagonal} with $m=1$ ($n=2$) 
we first get
	$$ N_{1 \times s,q^4} (q+1) = \frac{1}{q^{2(2-2s)}} \sum_{j=0}^{2}  \frac{(-1)^{js}}{q^{js}} \, Q_j(2) 
	= q^{4(s-1)} \big\{ 1 + \tfrac{(-1)^{s}}{q^s} Q_1(2) + \tfrac{1}{q^{2s}} Q_2(2) \big\}. $$
Now, by using \eqref{eq: Qj(n)} or Lemma \ref{lem: recursion Mjs} we obtain
	$$ 
	Q_1(2) = (q-1)(q^2+1) \qquad \text{and} \qquad Q_2(2) = q Q_1(2). 
	$$ 
In this way we get 
	$$ N_{1 \times s,q^4} (q+1) = q^{4(s-1)} \big\{ 1 + \tfrac{1}{q^s} Q_1(2) ((-1)^{s}+ \tfrac{q}{q^s}) \big\},$$
from where after some straightforward computations we get \eqref{eq: N1s}. The remaining assertion is clear.
\end{proof}

\noindent
\textit{Note}.
Taking $s=1$ in \eqref{eq: N1s} we get $N_{1 \times s, q^4}(q+1)=1$ as it should be.

We now look at examples for the smallest values of $q=2,3,5$ and any $s$.
\begin{exam}[$q=2$]
The number of solutions of $X_1^3+\cdots +X_s^3=0$ is given by 	
	$$  N_{1 \times s,2^4}(3) = 2^{2s-3} \big\{ 2^{2s-1} + 5(2^{s-1}(-1)^s +1)\big\}. $$
	In particular, $N_{1 \times 2,2^4}(3)= 46$, $N_{1 \times 3,2^4}(3)= 136$ and $N_{1 \times 4,2^4}(3)= 5{.}536$.
	\hfill $\diamond$
\end{exam} 

\begin{exam}[$q=3$]
We have that the number of solutions of $X_1^4+\cdots +X_s^4=0$ is  	
	$$  N_{1 \times s,3^4}(4) = 3^{2s-3} \{ 3^{2s-1} + 20(3^{s-1}(-1)^{s}+1) \}. $$
In particular, $N_{1 \times 2,3^4}(4)= 321$, $N_{1 \times 3,3^4}(4)= 2{.}241$ and $N_{1 \times 4,3^4}(4)= 667{.}521$.
\hfill $\diamond$
\end{exam}

\begin{exam}[$q=5$]
The number of solutions of $X_1^6+\cdots +X_s^6=0$ is  	
	$$  N_{1 \times s,5^4}(6) = 5^{2s-3} \big\{ 5^{2s-1} + 104(5^{s-1}(-1)^s +1)\big\}. $$
	In particular, $N_{1 \times 2,5^4}(6)= 3{.}745$, $N_{1 \times 3,5^4}(6)= 78{.}625$ and $N_{1 \times 4,5^4}(6)=285{.}090{.}625$.
	\hfill $\diamond$
\end{exam}

Now, we give explicit general formulas for the smallest values of $s$. That is, we present the first polynomials giving the number $N_{1\times s,q^4}(q+1)$ for $s=1,\ldots,9$.

\begin{coro} \label{coro: N1sq}
	The number of solutions $N_1(s,q)=N_{1\times s,q^4}(q+1)$ of the Hermitian-form diagonal equation $X_1^{q+1} + \cdots + X_s^{q+1}=0$ are given by
	\begin{align*} 
		N_1(1,q) &= 1, \\ 
		N_1(2,q) &= q(q^4 + q^3 - 1), \\ 
		N_1(3,q) &= q^3(q^4 + q - 1), \\ 
		N_1(4,q) &= q^5(q^7 + q^6 - q^5 + q^4 - q^2 + q - 1), \\ N_1(5,q) &= q^7(q^9 - q^7 + q^6 - q^5 + q^4 + q^3 - q^2 + q - 1), \\ 
		N_1(6,q) &= q^9(q^{11} + q^8 - q^7 + q^6 - q^5 + q^3 - q^2 + q - 1), \\ 
		N_1(7,q) &= q^{11}(q^{13} - q^9 + q^8 - q^7 + q^6 + q^3 - q^2 + q - 1), \\ 
		N_1(8,q) &= q^{13}(q^{15} + q^{10} - q^9 + q^8 - q^7 + q^3 - q^2 + q - 1), \\ 
		N_1(9,q) &= q^{15}(q^{17} - q^{11} + q^{10} - q^9 + q^8 + q^3 - q^2 + q - 1). 
	\end{align*}
\end{coro}

Notice that from $s=2$ on, the exponent of the isolated $q$ in the expressions above is odd. 
In fact, it seems that  
$ N_1(k,q) = q^{2k-3} P_k(q) $
for some polynomial $P_k(q) \in \Z[q]$ for any $k\ge 2$. 
So, we conjecture that 
$$ N_1(k+1,q) \equiv 0 \pmod{q^{2k-1}}$$
for any $k \in \N$, thus improving a lot the assertion in Proposition \ref{prop: N_1xs}.

\subsection{Hermitian-form systems of 2 diagonal equations}
Thus, we now give the number of solutions of the system $\mathcal{S}_{2\times s, q}((q+1,q^3+1))$ of 2 equations of the form
\begin{equation} \label{eq: m=2}
	\begin{cases}
		X_1^{q+1}+\cdots+ X_s^{q+1} =  0,  \\[1.5mm] 
		X_1^{q^3+1}+\cdots+ X_s^{q^3+1} =  0.
	\end{cases}
\end{equation}

We first give an addition formula for the number of solutions of these systems.
\begin{lem} \label{lema: N_2xs}
	Let $q$ be a prime power and $s\in \N$.
	The number of solutions in $(\ff_{q^8})^s$ of the Hermitian-form system of $2$ diagonal equations \eqref{eq: m=2}
	is given by
	\begin{equation} \label{eq: N2s}
		N_{2 \times s,q^8}((q+1,q^3+1)) = \frac{1}{q^{4(4-2s)}} \Big(1+ \sum_{j=1}^{4}  (-q)^{-js} \, Q_j(q,4) \Big) 
	\end{equation}
	where the $Q_1(q,4), Q_2(q,4), Q_3(q,4), Q_4(q,4)$ are as given in \eqref{eq: Qj(n)}.
\end{lem} 	

\begin{proof}
	Just apply Theorem \ref{teo: hermitiano diagonal} with $m=2$.
\end{proof}

Let us see some small examples.

\begin{exam}
	Let $q=2$ and $m=2$, hence $n=4$. The above theorem 
	allows us to compute the number of solutions $(x_1,\ldots,x_s)\in (\ff_{2^8})^{s}$ of the system $\mathcal{S}_{2\times s, 2}((3,9))$
	\begin{equation}\label{eq: q2 m4}
		\begin{cases}
			X_1^{3}+\cdots+ X_s^{3} =  0, \\[1.5mm]
			X_1^{9}+\cdots+ X_s^{9} = 0.
		\end{cases}
	\end{equation}
	In this case, we have $Q_{0}(2,4) = 1$ and, using Lemma \ref{lem: recursion Mjs} for instance, that
	$$ Q_{1}(2,4) = 85, \qquad Q_{2}(2,4) = 3{.}570, \qquad Q_{3}(2,4) = 23{.}800, \qquad Q_{4}(2,4) = 38{.}080.$$
	Therefore
	$$ N_{2 \times s,2^8}((3,9)) = 2^{4s-16} \big(2^{4s} + (-1)^{s}85 \cdot 2^{3s}+3{.}570 \cdot 2^{2s}+(-1)^s 23{.}800 \cdot2^{s}+38{.}080 \big).$$
	The first values of $N_{2 \times s,2^8}((3,9))$ are given in the next table: 
	\renewcommand{\arraystretch}{1.35}
	\begin{table}[H] \label{table}
		\begin{tabular}{|c||c|c|c|c|c|c|}
			\hline
			$s$   & $1$ & $2$ & $3$ & $4$ & $5$ & $6$\\\hline
			$N_{2 \times s,2^8}((3,9)) $ & $1$ & $766$ &  $2{.}296$ &  $1{.}746{.}496$ & $19{.}127{.}296$ & $14{.}142{.}324{.}736$\\ \hline
		\end{tabular}
		\caption{The number of solution of \eqref{eq: q2 m4} for small values of $s$.} 
	\end{table}
\end{exam}

\begin{exam}
	Let $q=3$ and $n=4$. The above theorem allows us to compute the number of solutions $(x_1,\ldots,x_s)\in (\ff_{3^8})^s$ of the system
	\begin{equation}\label{eq: q3 m4}
		\begin{cases}
			X_1^{4}+\cdots+ X_s^{4} =  0, \\[1.5mm]
			X_1^{28}+\cdots+ X_s^{28} = 0.
		\end{cases}
	\end{equation}
	In this case, using \eqref{eq: recursion for Mjs}
	we obtain
	$$ Q_{1}(3,4) = 1{.}640, \quad Q_{2}(3,4) =447{.}720, \quad Q_{3}(3,4) = 11{.}512{.}800, \quad Q_{4}(3,4) = 31{.}084{.}560.$$
	Therefore, by \eqref{eq: Nsm}, we get 
	$$ N_{2,s} = 3^{4s-16} \big(3^{4s}+(-1)^{s} 1{.}640 \cdot 3^{3s} + 447{.}720 \cdot 3^{2s}+(-1)^s 11{.}512{.}800 \cdot 3^{s}+31{.}084{.}560 \big)$$
	where $N_{2,s} = N_{2 \times s,3^8}((4,28))$.
	The first values of $N_{2 \times s,3^8}((4,28))$ are given in next table: 
	\renewcommand{\arraystretch}{1.35}
	\begin{table}[H] \label{table2}
		\begin{tabular}{|c||c|c|c|c|c|}
			\hline
			$s$   & $1$ & $2$ & $3$ & $4$ & $5$\\\hline
			$N_{2 \times s,3^8}((4,28))$ & $1$& $26{.}241$ &  $183{.}681$ & $4{.}815{.}722{.}241$ & $293{.}663{.}018{.}241$ \\ \hline
		\end{tabular} 
		\caption{The number of solution of \eqref{eq: q3 m4} for small values of $s$.} 	
	\end{table}
\end{exam}

Now, we give a closed formula for the number $N_{2 \times s,q^8}((q+1,q^3+1))$.

\begin{prop} \label{prop: N_2xs}
	Let $q$ be a prime power and $s\in \N$.
	The number of solutions in $(\ff_{q^8})^s$ of the Hermitian-form system of $2$ diagonal equations \eqref{eq: m=2}
	is given by
	\begin{equation} \label{eq: N2s 2}
		N_{2 \times s,q^8}((q+1,q^3+1)) =  
		q^{4s-11} \big\{ q^{4s-5} + (q-1)(q^2+1)(q^4+1) F_{s}(q) \big\},
	\end{equation}
	where 
	$$ F_{s}(q) = (-1)^s q^{3s-5} + (q^4+q^2+1)q^{2s-4} + (-1)^s (q^2+1)(q^3-1)q^{s-2} + q(q^3-1). $$
	In particular, $N_{2 \times s,q^8} ((q+1,q^3+1)) \equiv 0 \pmod q$ for every $s\ge 2$.
\end{prop} 

\begin{proof}
Let us use the abbreviation $N_{2,s} = N_{2 \times s,q^8}((q+1,q^3+1))$. We use Lemma \ref{lema: N_2xs}	and the recursive expressions for $Q_j(q,4)$ to put all in terms of $Q_1(q,4)$. 
After some tedious computations, we get to 
	$$	N_{2,s} =  
	q^{8s-16} \left\{ 1 + \frac{q^8-1}{q+1} \left( \frac{(-1)^s}{q^s} + \frac{q^4+q^2+1}{q^{2s-1}} + \frac{(-1)^s (q^2+1)(q^3-1)}{q^{3s-3}} + \frac{q^3-1}{q^{4s-6}} \right) \right\}.$$
Now, we express this a rational function in $q$ times a polynomial in $q$ (in brackets), 
getting 
	$$ N_{2,s} = \tfrac{q^{4s-11}}{q+1} \left\{  (q+1)q^{4s-5} + (q^8-1) P_s(q)  \right\}$$
where 
	$$ P_s(q) = (-1)^s q^{3s-5} + (q^4+q^2+1)q^{2s-4} + (-1)^s (q^2+1)(q^3-1)q^{s-2} + q(q^3-1). $$
From here the final expression \eqref{eq: N2s 2} in the statement is obtained, by factoring $q^8-1$. We leave the details to the reader. 
		
Now, from \eqref{eq: N2s 2}, the last assertion is clear for $s\ge 3$. For $s=2$, one can check with some effort that $N_{2\times s,q^8}((q+1,q^3+1)) = q(q^8 + q^7 - 1)$ from where the result follows. 
\end{proof}

To finish this subsection, we give the first polynomials counting the number of solutions explicitly.
\begin{coro} \label{coro: N2sq}
	The numbers $N_2(s,q) = N_{2\times s, q^8}((q+1,q^3+1))$ for 
	$s=2,3,4,5$ are given by 
	\begin{align*} 
		N_2(1,q) &= 1, \\ 
		N_2(2,q) &= q(q^8 + q^7 - 1), \\ 
		N_2(3,q) &= q^3(q^8 + q^5 - 1), \\ 
		N_2(4,q) &= q^6(q^{14} + q^{13} + q^{11} + q^{10} - q^8 - q^6 - q^5 - q^3 + 1), \\
		N_2(5,q) &= q^{10}(q^{14} + q^{11} + q^9 - q^8 - q^3 - q + 1), \\
		N_2(6,q) &= q^{14}(q^{19} + q^{18} + q^{16} + q^{14} - q^{12} - q^{10} + q^9 - 2q^8 - q^6 + q^4 - q^3 + q^2 - q + 1).
	\end{align*}
	and for $s=6, 7, 8, 9$ by 
	\begin{align*} 
		N_2(6,q) &= q^{14}(q^{19} + q^{18} + q^{16} + q^{14} - q^{12} - q^{10} + q^9 - 2q^8 - q^6 + q^4 - q^3 + q^2 - q + 1), \\ 
		N_2(7,q) &= q^{18}(q^{21} + q^{18} - q^{17} + q^{16} - q^{15} + q^{14} - q^{12} + 2q^{11} - 2q^{10} + 2q^9 - 2q^8 + q^7 \\ & \quad \qquad - q^5 + q^4 - 2q^3 + q^2 - q + 1), \\ 
		N_2(8,q) &= q^{22}(q^{26} + q^{25} - q^{24} + q^{23} + q^{20} - q^{19} + 2q^{18} - 2q^{17} + 2q^{16} - q^{15} - q^{14} + q^{13} \\ & \quad \qquad - 3q^{12} + 3q^{11} - 3q^{10} + 2q^9 - 2q^8 + q^6 - q^5 + 2q^4 - 2q^3 + q^2 - q + 1), \\ 
		N_2(9,q) &= q^{26}(q^{30} - q^{28} + q^{27} - q^{26} + q^{25} + q^{22} - q^{21} + 3q^{20} - 3q^{19} + 2q^{18} - 2q^{17} + q^{15} \\ & \quad \qquad - 2q^{14} + 3q^{13} - 4q^{12} + 4q^{11} - 2q^{10} + 2q^9 - q^8 - q^7 + q^6 - 2q^5 + 2q^4 \\ & \quad \qquad - 2q^3 + q^2 - q + 1). 
	\end{align*}
\end{coro}

\begin{rem}
	We make some comments on the exponent of the isolated $q$ in the expressions above. 
	
	\noindent ($i$) 
	Notice that from $s=4$ on, this exponent is of the form $q^{4k+2}$. 
	In fact, it seems that  
	$ N_2(k+3,q) = q^{4k+2} P_k(q) $
	for some polynomial $P_k(q) \in \Z[q]$. 
	So, we conjecture that 
	$$ N_2(k+3,q) \equiv 0 \pmod{q^{4k+2}}$$
	for any $k \in \N$, thus improving a lot the assertion in Proposition \ref{prop: N_2xs}.
	
	\noindent ($ii$)
	The numbers $1, 3, 6, 10, 14, 18, 22$, etc are exactly the coefficients of the expansion of 
	$ \frac{1 + x + x^2 + x^3}{(1-x)^2}$.
\end{rem}

\subsection{Hermitian-form systems of 3 diagonal equations}
We now give the number 
$$N_3(s,q) := N_{3 \times s,q^{12}}((q+1,q^3+1,q^5+1))$$
of solutions of the system $\mathcal{S}_{2\times s, q}((q+1,q^3+1, q^5+1))$ of 3 equations of the form
\begin{equation} \label{eq: m=3}
	\begin{cases}
		X_1^{q+1}   +\cdots+ X_s^{q+1}   =  0, \\[1.5mm] 
		X_1^{q^3+1} +\cdots+ X_s^{q^3+1} =  0, \\[1.5mm]
		X_1^{q^5+1} +\cdots+ X_s^{q^5+1} =  0.
	\end{cases}
\end{equation}

\begin{prop} \label{prop: N_3xs}
	Let $q$ be a prime power and $s\in \N$.
	The number of solutions in $(\ff_{q^{12}})^s$ of the Hermitian-form system of $3$ diagonal equations \eqref{eq: m=2}
	is given by
	\begin{equation} \label{eq: N3s}
		N_{3 \times s,q^{12}}((q+1,q^3+1,q^5+1)) = \frac{1}{q^{6(6-2s)}} \sum_{j=0}^{6}  (-q)^{-js} \, Q_j(6,q), 
	\end{equation}
	where the numbers $Q_1(6,q), \ldots, Q_6(6,q)$ are given as follows: 
	\begin{align*} 
		Q_1(6,q) &= \frac{q^{12}-1}{q+1}, \\[1.25mm] 
		Q_2(6,q) &= q \, \frac{q^{12}-1}{q+1} (q^8+q^6+q^4+q^2+1), \\[1.25mm] 
		Q_3(6,q) &= q^3 \, \frac{q^8-1}{q+1} (q^9-q^6+q^3-1) (q^8+q^6+q^4+q^2+1), \\[1.25mm] 
		Q_4(6,q) &= q^6 \, \frac{q^8-1}{q+1} (q^9-q^6+q^3-1) (q^8+q^6+q^4+q^2+1) (q^2+1), \\[1.25mm] 
		Q_5(6,q) &= q^{10} \, \frac{(q^{12}-1)(q^8-1)(q^6-1)(q^4-1)}{(q+1)(q^3+1)(q^5+1)} (q^2+1) (q^8+q^6+q^4+q^2+1), \\[1.25mm] 
		Q_6(6,q) &= q^{15} \,  \frac{(q^{12}-1)(q^8-1)(q^6-1)(q^4-1)(q^2-1)}{(q+1)(q^3+1)(q^5+1)} (q^8+q^6+q^4+q^2+1).
	\end{align*}
\end{prop} 

\begin{proof}
	Just apply Theorem \ref{teo: hermitiano diagonal} with $m=3$
	and then Lemma \ref{lem: Q1,..,Q6} to obtain the expressions for $Q_1(6,q), \ldots, Q_6(6,q)$. 
\end{proof}

In this case it would be too cumbersome to give a closed expression for $N_3(s,q)$ with this method, as we did for the case of 1 and 2 equations in Propositions \ref{prop: N_1xs} and \ref{prop: N_2xs}, respectively, showing the complexity of the problem.
However, we can give closed formulas for the first cases.

\begin{coro}
	The numbers $N_3(s,q) = N_{3\times s, q^{12}}((q+1,q^3+1,q^5+1))$ for the first values of $s$ are given as follows. For $s=1,2,3$ we have the closed formulas
	\begin{align*} 
		N_3(1,q) &= 1, \\[1.5mm] 
		N_3(2,q) &= q(q^{12} + q^{11} - 1), \\[1.5mm] 
		N_3(3,q) &= q^3(q^{12} + q^9 - 1), 
	\end{align*}
	while for $s=4,5$ we obtain the expressions
	\begin{align*} 
		N_3(4,q) &= q^{12} \left\{ 1 + \tfrac{Q_1(6,q)}{q^4} + \tfrac{Q_2(6,q)}{q^8} + \tfrac{Q_3(6,q)}{q^{12}} + \tfrac{Q_4(6,q)}{q^{16}} + \tfrac{Q_5(6,q)}{q^{20}} + \tfrac{Q_6(6,q)}{q^{24}} \right\}, \\[1.5mm] 
		N_3(5,q) &= q^{24} \left\{ 1 - \tfrac{Q_1(6,q)}{q^5} + \tfrac{Q_2(6,q)}{q^{10}} - \tfrac{Q_3(6,q)}{q^{15}} + \tfrac{Q_4(6,q)}{q^{20}} - \tfrac{Q_5(6,q)}{q^{25}} + \tfrac{Q_6(6,q)}{q^{30}} \right\}, 
	\end{align*}
	where $Q_1(6,q), \ldots, Q_6(6,q)$ are given in Proposition \ref{prop: N_3xs}.
\end{coro}

To illustrate the results, we give the following table with $N_3(2,q)$ and $N_3(3,q)$ for the smallest values of $q$:
\renewcommand{\arraystretch}{1.2}
\begin{table}[H]
	\centering
	\begin{tabular}{ccc}
		\hline
		$q$ & $s=2$ & $s=3$ \\
		\hline
		2 &  12,286 & 36,856 \\
		3 &  2,125,761 & 14,880,321 \\
		4 &  83,886,076 & 1,090,518,976 \\
		5 &  1,464,843,745 & 30,761,718,625 \\
		7 &  110,730,297,601 & 4,761,402,796,801 \\
		8 &  618,475,290,616 & 35,253,091,565,056 \\
		9 &  2,824,295,364,801 & 206,173,561,630,401 \\
		\hline
	\end{tabular}
	
	\caption{Values of $N_3(s,q)$ for the smallest values of $s$ and $q$.}
	\label{tab:valores_N3}
\end{table}

\section{Any number of equations and variables} \label{sec: caso gral}
In this section, we consider the general case. First, we show that we can give a simple closed formula for the number of solutions to the homogeneous system of monic Hermitian-form diagonal equations for any number $m$ of equations as long as we keep the number of variables really small, namely up to three. Then, we will obtain a 2-term recursion for $N_m(s,q)$ in terms of $N_m(s-1,q)$ and $N_m(s-2,q)$ for any $m, s$. Finally we obtain a neat expression of $N_m(s,q)$ in terms of a certain $q$-hypergeometric function $_2\phi_0 (a, b; -; q, z)$.
As in the previous section, we will use the shorthand notation $N_m(s,q)$ for $N_{m \times s, q^{4m}}((q+1,q^3+1, \ldots, q^{2m-1}+1))$. 

\subsection*{The functional equation}
To this end, we first give a functional equation for the finite generating function of the numbers $Q_j(q,2m)$, 
\begin{equation} \label{eq: gen fn of Qjs}
F(y) := F_{m,s,q}(y) = \sum_{j=0}^n y^j Q_j(q,2m).	
\end{equation}
Observe that the expression for the number $N_{m}(s,q)$ of solutions of the homogeneous monic Hermitian-form SDEs obtained \eqref{eq: Nsm} of Theorem \ref{teo: hermitiano diagonal} is related with $F(y)$, in fact
\begin{equation}
	N_m(s,q) = \tfrac{1}{q^{n^2-2ns}} F((-q)^{-s}).
\end{equation}

\begin{lem}
The polynomial generating function $F(y)$ of the numbers $Q_j(q,2m)$ satisfy the following functional equation
\begin{equation} \label{eq: functional eqn}
	(1+y) F(yq) - F(-y) = q^{2n} y F(\tfrac yq).
\end{equation}
\end{lem}

\begin{proof}
Let $n = 2m$, to simplify the notation.
We start from the recurrence relation  
$$ (q^{j+1} + (-1)^j) Q_{j+1} = (q^{2n-j} - q^j) Q_j$$
in Lemma \ref{lem: recursion Mjs} 
Multiplying both sides by $y^{j+1}$ and adding from $j=0$ to $n-1$ we get
$$	\sum_{j=0}^{n-1} y^{j+1} q^{j+1} Q_{j+1} + \sum_{j=0}^{n-1} y^{j+1} (-1)^j Q_{j+1} = \sum_{j=0}^{n-1} y^{j+1} q^{2n-j} Q_j - \sum_{j=0}^{n-1} y^{j+1} q^j Q_j. $$

We perform an index shift $k = j+1$ on the left side. Furthermore, we notice two key facts: ($a$) since $(yq)^0 - (-y)^0 = 1 - 1 = 0$, we can include the $k=0$ term on the left side for free and ($b$) since $Q_{n+1} = 0$, we can extend the sums on the right side up to $j=n$.

Now, rewriting everything in terms of our polynomial $F$, we obtain
$$\sum_{k=0}^n (yq)^k Q_k - \sum_{k=0}^n (-y)^k Q_k = y q^{2n} \sum_{j=0}^n \left(\tfrac{y}{q}\right)^j Q_j - y \sum_{j=0}^n (yq)^j Q_j.$$
This expression collapses into the functional equation
$$ F(yq) - F(-y) =  q^{2n} y F(\tfrac yq) - y F(yq), $$
from which expression \eqref{eq: functional eqn} readily follows.
\end{proof}

\subsection*{Explicit results for $s$ small}
The number $N_m(1,q)$ is trivially 1. 
Using the previous functional equation, we can now give explicit simple expressions for the numbers $N_m(s,q)$, $s=2, 3, 4, 5$, of solutions of homogeneous monic Hermitian-form SDEs of two, three, four and five variables respectively, in any number of equations. 

\begin{prop} \label{prop: Nm2q Nm3q}
For any integer $m \ge 1$ the number of solutions of the homogeneous monic Hermitian-form SDEs as in \eqref{eq: monic homogeneous Herm form diag Eqn} with two and three variables are respectively given by 
\begin{equation} \label{eq: Nm2q Nm3q}
\begin{aligned}
	N_m(2,q) &= q(q^{4m} + q^{4m-1} - 1), \\[1mm]
	N_m(3,q) &= q^3(q^{4m} + q^{4m-3} - 1),	
\end{aligned}	
\end{equation}
while for four and five variables they are respectively given by 
\begin{equation} \label{eq: Nm4q Nm5q}
\begin{aligned}
N_m(4,q) & = q^{8m}(q^4+q^3+q+1) - q^{4m}(q^6+q^4+q^3+q) + q^6, \\ 
N_m(5,q) & = q^{8m}(q^8+q^5+q^3+1) - q^{4m}(q^{10}+q^8+q^5+q^3) + q^{10}.
\end{aligned}	
\end{equation}
\end{prop}

\begin{proof}
We first compute the initial conditions.
Recall that $N_m(s,q)$ relates to $F(y)$ via
$$ N_m(s,q) = \frac{1}{q^{n(n-2s)}} \sum_{j=0}^n (-q)^{-js} Q_j(n,q) = q^{-n^2+2ns} F((-q)^{-s}). $$
For $s=1$, the result is trivially $N_m(1,q) = 1$.
Using this, we can deduce $F(-q^{-1})$:
$$ 1 = q^{-n^2+2n} F(-q^{-1}) \quad \Rightarrow \quad  F(-q^{-1}) = q^{n^2-2n}.$$
Now, we evaluate the functional equation \eqref{eq: functional eqn} at $y = -1$ obtaining
$$(1 - 1) F(-q) - F(1) = -1 \cdot q^{2n} F(-q^{-1})$$
from where we finally get
$$F(1) = q^{2n} \cdot q^{n^2-2n} = q^{n^2}.$$

\noindent 
\textit{The case $s=2$.}
To compute $N_m(2,q)$, we need to evaluate $F(y)$ at $y = (-q)^{-2} = q^{-2}$, i.e.
$$ N_m(2,q) = q^{-n^2+4n} F(q^{-2}).$$

Now, evaluating the functional equation \eqref{eq: functional eqn} at $y = q^{-1}$ we get
$$ (1 + q^{-1}) F(1) - F(-q^{-1}) = q^{-1} q^{2n} F(q^{-2}). $$
Substituting the values of $F(1)$ and $F(-q^{-1})$ that we have already calculated gives us
$$ (1 + q^{-1}) q^{n^2} - q^{n^2-2n} = q^{2n-1} F(q^{-2}). $$
Now, we solve for $F(q^{-2})$ by multiplying by $q^{-2n+1}$ and obtain
$$ F(q^{-2}) = q^{-2n+1} (q^{n^2} + q^{n^2-1} - q^{n^2-2n}) = q^{n^2-2n+1} + q^{n^2-2n} - q^{n^2-4n+1}. $$
Putting this into the formula for $N_m(2,q)$ we get
\begin{align*}
	 N_m(2,q) & = q^{-n^2+4n} (q^{n^2-2n+1} + q^{n^2-2n} - q^{n^2-4n+1}) \\ 
			  & = q^{2n+1} + q^{2n} - q^1 = q(q^{2n} + q^{2n-1} - 1).
\end{align*}
Since $n = 2m$, we get that $N_m(2,q) = q(q^{4m} + q^{4m-1} - 1)$, as desired.

\noindent 
\textit{The case $s=3$.}
For $s=3$, we have that 
$$ N_m(3,q) = q^{-n^2+6n} F(-q^{-3}).$$ 

To find $F(-q^{-3})$, we evaluate the functional equation \eqref{eq: functional eqn} at $y = -q^{-2}$. 
This specific choice shifts the arguments perfectly since 
$yq = -q^{-1}$, $-y = q^{-2}$ and $\frac yq = -q^{-3}$. 
Thus, substituting $y = -q^{-2}$ into the functional equation, we get:
$$(1 - q^{-2}) F(-q^{-1}) - F(q^{-2}) = -q^{2n-2} F(-q^{-3}).$$

From the case $s=2$, we already know the values for $F(-q^{-1})$ and $F(q^{-2})$: 
	$$ F(-q^{-1}) = q^{n^2-2n} \qquad \text{and} \qquad F(q^{-2}) = q^{n^2-2n+1} + q^{n^2-2n} - q^{n^2-4n+1}.$$ 
We plug these values into the left-hand side (LHS) of our equation obtaining
\begin{align*}
	\text{LHS} & = (1 - q^{-2}) q^{n^2-2n} - ( q^{n^2-2n+1} + q^{n^2-2n} - q^{n^2-4n+1})	\\[1.5mm]
			   & = - q^{n^2-2n+1} - q^{n^2-2n-2} + q^{n^2-4n+1}.
\end{align*} 
Now we equate this LHS to the right-hand side of the functional equation \eqref{eq: functional eqn} evaluated at $y = -q^{-2}$:$$-q^{2n-2} F(-q^{-3}) = - q^{n^2-2n+1} - q^{n^2-2n-2} + q^{n^2-4n+1}.$$ 
We isolate $F(-q^{-3})$ by dividing both sides by $-q^{2n-2}$ 
\begin{align*}
F(-q^{-3}) & = q^{(n^2-2n+1) - (2n-2)} + q^{(n^2-2n-2) - (2n-2)} - q^{(n^2-4n+1) - (2n-2)} \\
  		   & = q^{n^2-4n+3} + q^{n^2-4n} - q^{n^2-6n+3}.	
\end{align*}
Finally, we substitute this result back into our equation for $N_m(3,q)$: 
\begin{align*}
	N_m(3,q) & = q^{-n^2+6n} F(-q^{-3}) = q^{-n^2+6n} ( q^{n^2-4n+3} + q^{n^2-4n} - q^{n^2-6n+3}) \\
			 & = q^{2n+3} + q^{2n} - q^3 = q^3(q^{4m} + q^{4m-3} - 1), 	
\end{align*}
since $n = 2m$, 
as we wanted to show.

\noindent 
\textit{The cases $s=4, 5$.}
It is clear now that for the cases $s=4,5$ one proceeds similarly as above, now evaluating the functional equation at the values $y=(-q)^{-3}$ and $y=(-q)^{-4}$ respectively.
We leave the details to the interested reader. This completes the proof.
\end{proof}

It is clear that using the technique in the proof one can obtain more an more explicit expressions, with increasing difficulty, for the numbers $N_m(s,q)$. 
	
Notice also that the expressions obtained in \eqref{eq: Nm2q Nm3q} in Theorem \ref{eq: Nm2q Nm3q} coincide in the cases $m=1$ and $m=2$ with the previous results obtained in Corollaries \ref{coro: N1sq} and \ref{coro: N2sq}, respectively.

\subsection*{General recursive results}
We end this work with a recursive formula for the numbers $N_m(s,q)$. It turns out that the number of solutions of a system of $m$ equations with $s$ variables can be computed from the number of solutions of two systems with $m$ equations but with $s-1$ and $s-2$ variables.

\begin{thm} \label{thm: recursive Nmsq}
For any integer $m \ge 1$ and $s \ge 2$, the number of solutions $N_m(s,q)$ satisfies the following two-terms linear recurrence relation 
\begin{equation} \label{eq: two term recursion}
N_m(s,q) = (-q)^{s-1} N_m(s-1,q) + q^{4m} \big\{1 - (-q)^{s-1} \big\} N_m(s-2,q)
\end{equation}
with initial conditions $N_m(0,q) = 1$ and $N_m(1,q) = 1$.
\end{thm}

\begin{proof}
Let $n = 2m$. We define the sequence of evaluation points $y_k = (-q)^{-k}$. From the general formula relating $N_m(k,q)$ to our generating function $F(y)$, we have $N_m(k,q) = q^{-n^2+2nk} F(y_k)$ and hence 
	$$ F(y_k) = q^{n^2-2nk} N_m(k,q).$$ 
Let us write $N_k$ as a shorthand for $N_m(k,q)$. 
We consider the functional equation 
	$$ (1+y) F(yq) - F(-y) = y q^{2n} F(\tfrac yq).$$ 
Since we want to find an expression for the number $N_m(s,q)$, we evaluate this equation at $y = -y_{s-1} = -(-q)^{-(s-1)}$. 
Let us first compute the arguments of $F$ under this substitution 
\begin{gather*}
	yq = -(-q)^{-s+1} q = (-1)^s q^{-s+2} = (-q)^{-(s-2)} = y_{s-2}, \\ 
	-y = y_{s-1}, \qquad \tfrac yq = -(-q)^{-s+1} q^{-1} = (-q)^{-s} = y_s. 
\end{gather*}
Substituting these arguments into the functional equation yields 
	$$ (1 - y_{s-1}) F(y_{s-2}) - F(y_{s-1}) = -y_{s-1} q^{2n} F(y_s). $$ 
We now solve for the highest-order term, $F(y_s)$ obtaining 
	$$ F(y_s) = \frac{F(y_{s-1}) - (1 - y_{s-1}) F(y_{s-2})}{y_{s-1} q^{2n}}.$$ 
Next, we introduce the relation $F(y_k) = q^{n^2-2nk} N_k$ into the equation: 
$$q^{n^2-2ns} N_s = \frac{q^{n^2-2n(s-1)} N_{s-1} - (1 - (-q)^{-s+1}) q^{n^2-2n(s-2)} N_{s-2}}{(-q)^{-s+1} q^{2n}}$$ 
and hence we get
$$ q^{n^2-2ns} N_s = (-q)^{s-1} q^{-2n} q^{n^2-2ns} \Big\{ q^{2n} N_{s-1} - (1 - (-q)^{-s+1}) q^{4n} N_{s-2} \Big\}.$$
Dividing both sides by $q^{n^2-2ns}$ cancels out all the quadratic powers of $n$
\begin{align*}
N_s &= (-q)^{s-1} q^{-2n} \Big\{ q^{2n} N_{s-1} - (1 - (-q)^{-s+1}) q^{4n} N_{s-2} \Big\} \\
& = (-q)^{s-1} N_{s-1} - (-q)^{s-1} q^{2n} \big\{1 - (-q)^{-s+1}\big\} N_{s-2}.	
\end{align*}
Finally, using that $2n = 4m$ yields the final recurrence relation  
$$ N_m(s,q) = (-q)^{s-1} N_m(s-1,q) + q^{4m} \big(1 - (-q)^{s-1}\big) N_m(s-2,q) $$ 
as claimed. This completes the proof. 
\end{proof}

Now, applying this recursion, we can give the expressions for the first numbers $N_m(s,q)$. Next, we give just the first 12 values. We will use the shorthand notation 
	$$ F_k=q^{2k-1}+1 \qquad \text{for $k\in \N$.}$$ 
The expressions are governed by $k = \lfloor \tfrac s2 \rfloor$ and the $q^2$-binomial coefficients 
$$ {k \brack j}_{q^2} = \prod_{i=0}^{j-1} \frac{(q^2)^k-(q^2)^i}{(q^2)^j-(q^2)^i} = \frac{(q^{2k}-1)(q^{2k-2}-1)\cdots(q^{2(k-j+1)}-1)}{(q^{2j}-1)(q^{2j-2}-1)\cdots(q^2-1)}$$
also called Gaussian polynomials.
Recall that ${n \brack j}_{q^2}$ counts the number of subspaces of dimension $j$ in an $n$-dimensional space over $\ff_q$, that is it is the number of points in the finite Grassmannian $Gr(k,\ff_{q^n})$.

Using the previous recursion we can get explicit expressions for $N_m(s,q)$ for the first values of $s$. 
\begin{exam}
	The numbers $N_m(s,q)$ for $1\le s \le 7$ are given by 
	\begin{align*} 
		N_m(1,q) =\ & 1, \\[1mm] 
		N_m(2,q) =\ & q^{4m} F_1 - q, \\[1mm] 
		N_m(3,q) =\ & q^{4m} F_2 - q^3, \\[1mm] 
		N_m(4,q) =\ & q^{8m} F_1 F_2 - q^{4m} q \tcbinom{2}{1}_{q^2} F_2 + q^6, \\[1mm] 
		N_m(5,q) =\ & q^{8m} F_2 F_3 - q^{4m} q^3 \tcbinom{2}{1}_{q^2} F_3 + q^{10}, \\[1mm] 
		N_m(6,q) =\ & q^{12m} F_1 F_2 F_3 - q^{8m} q \tcbinom{3}{1}_{q^2} F_2 F_3 + q^{4m} q^6 \tcbinom{3}{2}_{q^2} F_3 - q^{15}, \\[1mm] 
		N_m(7,q) =\ & q^{12m} F_2 F_3 F_4 - q^{8m} q^3 \tcbinom{3}{1}_{q^2} F_3 F_4 + q^{4m} q^{10} \tcbinom{3}{2}_{q^2} F_4 - q^{21}.
	\end{align*}
	We can give more and more expressions but their complexity grows very fast.
	\end{exam}
	
	The expressions obtained above show many regularity, so it is not a surprise that we can get a general formula for $N_m(s,q)$. 

\subsection*{The general formula for $N_m(s,q)$}
We next show that the two-term linear recurrence with variable coefficients obtained in Theorem \ref{thm: recursive Nmsq}
admits an explicit closed-form solution in terms of Gaussian polynomials. 
We will need to separate the cases $s$ even and $s$ odd. This is one of the main results in the paper. 
 
\begin{thm} \label{thm: sums for Nmsq}
	For any $m, k \in \N$ we have 
	\begin{equation} \label{eq: Nmsq final sums}
		\begin{aligned}
			N_m(2k,q) 	&= (-1)^{k}q^{k(2k-1)} + \sum_{j=0}^{k-1} (-1)^j q^{j(2j-1)} q^{4m(k-j)} \tcbinom{k}{j}_{q^2} \Big( \prod_{i=j+1}^k F_i \Big) , \\[1.5mm]
			N_m(2k+1,q) &= (-1)^{k}q^{k(2k+1)} + \sum_{j=0}^{k-1} (-1)^j q^{j(2j+1)} q^{4m(k-j)} \tcbinom{k}{j}_{q^2} \Big( \prod_{i=j+1}^k F_{i+1} \Big).
		\end{aligned}
	\end{equation}
\end{thm}

\begin{proof}
	We proceed by induction on $s$. 
	The base cases $s=1$ and $s=2$ follow trivially by evaluating the sum for $k=0$ and $k=1$, which yield 
	$N_m(1,q)=1$ 
	and $N_m(2,q) = q^{4m}F_1 - q$. 
	
	Assume the formulas hold for all values of $s$ up to $s-1$. We evaluate the recurrence relation for the step $s$. Because the parity of $s$ alternates, the proof splits naturally into the transition from odd-to-even and even-to-odd. 
	
	Suppose $s = 2k$ is even. By the recurrence hypothesis we have
	$$ N_{m}(2k,q) = q^{4m} F_{k} N_m(2k-2,q) -q^{2k-1} N_m(2k-1,q),$$ 
	where $F_k=(1 + q^{2k-1})$.
	
	By induction, we have that
	$$ \begin{aligned}
		N_m(2k-2,q) &= \sum_{j=0}^{k-2} (-1)^j q^{j(2j-1)} q^{4m(k-1-j)} \tcbinom{k-1}{j}_{q^2} \Big( \prod_{i=j+1}^{k-1} F_i \Big)+(-1)^{k-1}q^{(k-1)(2k-3)}, \\[1.5mm]
		N_m(2k-1,q) &= \sum_{j=0}^{k-2} (-1)^j q^{j(2j+1)} q^{4m(k-1-j)} \tcbinom{k-1}{j}_{q^2} \Big( \prod_{i=j+1}^{k-1} F_{i+1} \Big)+(-1)^{k-1}q^{(k-1)(2k-1)}.
	\end{aligned} $$
	Hence, putting these expressions in the previous recurrence we have
	$$ \begin{aligned}
		N_{m}(2k,q) 
		& = \sum_{j=0}^{k-2} (-1)^j q^{j(2j-1)} q^{4m(k-j)} \tcbinom{k-1}{j}_{q^2} \Big( \prod_{i=j+1}^{k} F_i \Big)+(-1)^{k-1}q^{(k-1)(2k-3)}q^{4m}F_k \\ 
		& -q^{2k-1}\sum_{j=0}^{k-2} (-1)^j q^{j(2j+1)} q^{4m(k-1-j)} \tcbinom{k-1}{j}_{q^2} \Big( \prod_{i=j+1}^{k-1} F_{i+1} \Big) \\ & +(-1)^{k-1}q^{(k-1)(2k-1)}(-q^{2k-1}) \\
		& = q^{4mk}\Big( \prod_{i=1}^{k} F_i \Big) + \sum_{j=1}^{k-2} (-1)^j q^{j(2j-1)} q^{4m(k-j)} \tcbinom{k-1}{j}_{q^2} \Big( \prod_{i=j+1}^{k} F_i \Big) \\ & + (-1)^{k-1}q^{(k-1)(2k-3)}q^{4m}F_k -q^{2k-1}(-1)^{k-2} q^{(k-2)(2k-3)} q^{4m} \tcbinom{k-1}{k-2}_{q^2} F_k \\ & + (-1)^{k}q^{k(2k-1)} - q^{2k-1} \sum_{j=0}^{k-3} (-1)^j q^{j(2j+1)} q^{4m(k-1-j)} \tcbinom{k-1}{j}_{q^2} \Big( \prod_{i=j+1}^{k-1} F_{i+1} \Big) \\
	\end{aligned} $$
	
	Notice that the expression
	$$
	(-1)^{k-1}q^{(k-1)(2k-3)}q^{4m}F_k-q^{2k-1}(-1)^{k-2} q^{(k-2)(2k-3)} q^{4m} \tcbinom{k-1}{k-2}_{q^2} F_k 
	$$ 
	is equal to 
	$$ 
	(-1)^{k-1}q^{(k-1)(2k-3)}q^{4m}F_k\Big(\tcbinom{k-1}{k-1}_{q^2} +q^{2}\tcbinom{k-1}{k-2}_{q^2} \Big)
	$$
	which simplifies to
	$$
	(-1)^{k-1}q^{(k-1)(2k-3)}q^{4m} \tcbinom{k}{k-1}_{q^2}F_k.
	$$
	
	On the other hand, by a simple change of variables
	$$ q^{2k-1} \sum_{j=0}^{k-3} (-1)^j q^{j(2j+1)} q^{4m(k-1-j)} \tcbinom{k-1}{j}_{q^2} \Big( \prod_{i=j+1}^{k-1} F_{i+1} \Big) $$ 
	equals 
	$$ \sum_{j=1}^{k-2} (-1)^{j-1} q^{(j-1)(2j-1)+2k-1} q^{4m(k-j)} \tcbinom{k-1}{j-1}_{q^2} \Big( \prod_{i=j}^{k-1} F_{i+1} \Big).$$
	By taking into account that $\prod_{i=j}^{k-1} F_{i+1}=\prod_{i=j+1}^{k} F_{i}$, we have that
	$$ q^{2k-1} \sum_{j=0}^{k-3} (-1)^j q^{j(2j+1)} q^{4m(k-1-j)} \tcbinom{k-1}{j}_{q^2} \Big( \prod_{i=j+1}^{k-1} F_{i+1} \Big)$$ 
	is equal to 
	$$ \sum_{j=1}^{k-2} (-1)^{j-1} q^{j(2j-1)+2k-2j} q^{4m(k-j)} \tcbinom{k-1}{j-1}_{q^2} \Big( \prod_{i=j+1}^{k} F_{i} \Big). $$
	Hence, the difference $S=S_1-S_2$ between the sums
	\begin{align*}
	& 	S_1 = \sum_{j=1}^{k-2} (-1)^j q^{j(2j-1)} q^{4m(k-j)} \tcbinom{k-1}{j}_{q^2} \Big( \prod_{i=j+1}^{k} F_i \Big), \\ 
	&	S_2 = -q^{2k-1} \sum_{j=0}^{k-3} (-1)^j q^{j(2j+1)} q^{4m(k-1-j)} \tcbinom{k-1}{j}_{q^2} \Big( \prod_{i=j+1}^{k-1} F_{i+1} \Big),
	\end{align*}
	simplifies to
	$$
	S= \sum_{j=1}^{k-2} (-1)^{j} q^{j(2j-1)} q^{4m(k-j)}  \Big( \prod_{i=j+1}^{k} F_{i} \Big) \Big( \tcbinom{k-1}{j}_{q^2}+q^{2k-2j} \tcbinom{k-1}{j-1}_{q^2} \Big).
	$$
	Recall that the bracketed coefficients satisfy the fundamental $q$-Pascal identity 
	$$
	{k \brack j}_{q^2} = {k-1 \brack j}_{q^2} + q^{2(k-j)} {k-1 \brack j-1}_{q^2}.
	$$
	Thus, we obtain that the difference is equal to 
	$$
	\sum_{j=1}^{k-2} (-1)^{j} q^{j(2j-1)} q^{4m(k-j)} \tcbinom{k}{j}_{q^2} \Big( \prod_{i=j+1}^{k} F_{i} \Big)
	$$
	Therefore, we have that 
	$$
	\begin{aligned}
		N_{m}(2k,q)& = q^{4mk}\Big( \prod_{i=1}^{k} F_i \Big)+ \sum_{j=1}^{k-2} (-1)^{j} q^{j(2j-1)} q^{4m(k-j)} \tcbinom{k}{j}_{q^2} \Big( \prod_{i=j+1}^{k} F_{i} \Big)  \\
		& +   (-1)^{k-1}q^{(k-1)(2k-3)}q^{4m} \tcbinom{k}{k-1}_{q^2}F_k + (-1)^{k}q^{k(2k-1)} \\
		&= \sum_{j=0}^{k-1} (-1)^{j} q^{j(2j-1)} q^{4m(k-j)} \tcbinom{k}{j}_{q^2} \Big( \prod_{i=j+1}^{k} F_{i} \Big) + (-1)^{k}q^{k(2k-1)},
	\end{aligned}
	$$
	as asserted. 
	
	The case of $s$ odd is completely analogous to the case $s$ even and we leave the details. 
\end{proof}

\begin{rem}
	Using the general formulas in the previous proposition one can check that the final expressions for the polynomials $N_m(s,q)$ for $s=2, 3, 4, 5$ are the same as the ones obtained in Proposition~\ref{prop: Nm2q Nm3q}. 
	For $s=6, 7, 8, 9$ we get 
	$$\begin{aligned} 
		N_m(6,q) &= q^{12m}(q+1)(q^3+1)(q^5+1) - q^{8m+1}(q^4+q^2+1)(q^3+1)(q^5+1) \\ &\quad + q^{4m+6}(q^4+q^2+1)(q^5+1) - q^{15}, \\ N_m(7,q) &= q^{12m}(q^3+1)(q^5+1)(q^7+1) - q^{8m+3}(q^4+q^2+1)(q^5+1)(q^7+1) \\ &\quad + q^{4m+10}(q^4+q^2+1)(q^7+1) - q^{21}, \\ N_m(8,q) &= q^{16m}(q+1)(q^3+1)(q^5+1)(q^7+1) \\ &\quad - q^{12m+1}(q^6+q^4+q^2+1)(q^3+1)(q^5+1)(q^7+1) \\ &\quad + q^{8m+6}(q^8+q^6+2q^4+q^2+1)(q^5+1)(q^7+1) \\ &\quad - q^{4m+15}(q^6+q^4+q^2+1)(q^7+1) + q^{28}, \\ 
		N_m(9,q) &= q^{16m}(q^3+1)(q^5+1)(q^7+1)(q^9+1) \\ &\quad - q^{12m+3}(q^6+q^4+q^2+1)(q^5+1)(q^7+1)(q^9+1) \\ &\quad + q^{8m+10}(q^8+q^6+2q^4+q^2+1)(q^7+1)(q^9+1) \\ &\quad - q^{4m+21}(q^6+q^4+q^2+1)(q^9+1) + q^{36}. 
	\end{aligned}$$
	We leave the details of these computations.	
\end{rem}

\section{The hypergeometric connection} \label{sec: hypergeometric}
It is highly instructive to contextualize the algebraic nature of the generating polynomial $F(y)$ within the broader framework of special functions. 
In this final section, we show that $N_m(s,q)$ is a rational multiple of a $_2\phi_0$ hypergeometric series.

To this end, we need to recall first the standard definition of the basic hypergeometric series $_2\phi_0$ with base $q$ and argument $z$
	$$_2\phi_0 \Big( \begin{matrix} a, b \\ - \end{matrix} ; q, z \Big) = \sum_{j=0}^\infty \frac{(a; q)_j (b; q)_j}{(q; q)_j} \Big( (-1)^j q^{-\binom{j}{2}} \Big) z^j$$ 
where 
	$$ (x; q)_j = \prod_{k=0}^{j-1} (1-xq^k) $$ 
is the $q$-\textit{Pochhammer symbol}. The result is the following.

\begin{thm} \label{prop: hipergeometrica}
Let $m \ge 1$ be an integer and $n=2m$. For any integer $s \ge 1$, the number of solutions $N_m(s,q)$ can be explicitly expressed in terms of a basic hypergeometric series evaluated over the base $-q$ as
\begin{equation} \label{eq: NMsq Phi20}
	N_m(s,q) = q^{-n^2+2ns} \ _2\phi_0 \Big( \begin{matrix} q^{-n}, -q^{-n} \\ - \end{matrix} ; -q, -(-q)^{2n-s} \Big).	
\end{equation} 
\end{thm}

\begin{proof}
The ratio between the $(j+1)$-th term, denoted as $t_{j+1}$, and the $j$-th term, $t_j$, of this formal series is given by the rational function
	$$ \frac{t_{j+1}}{t_j} = \frac{(1 - a q^j)(1 - b q^j)}{1 - q^{j+1}} q^{-j} (-z).$$ 
	
Now, consider our generating polynomial $F(y) = \sum_{j=0}^n c_j$, where the $j$-th coefficient is $c_j = y^j Q_j(n,q)$. 
Using the explicit formula for $Q_j(n,q)$ and performing a change of base to $Q = -q$, the ratio of consecutive terms in our polynomial is 
	$$ \frac{c_{j+1}}{c_j} = y \frac{Q_{j+1}}{Q_j} = y (-Q)^j \frac{Q^{2n-2j}-1}{(-Q)^{j+1} + (-1)^j} = -y Q^j \frac{1 - Q^{2n-2j}}{1 - Q^{j+1}}.$$ 
We can factor the numerator using negative powers of $Q$ as a difference of squares 
	$$ 1 - Q^{2n-2j} = -Q^{2n-2j} (1 - Q^{-2n+2j}) = -Q^{2n-2j} (1 - Q^{-n}Q^j) (1 + Q^{-n}Q^j).$$ 
Substituting this factorization back into the ratio for $c_{j+1}/c_j$, we obtain
	$$ \frac{c_{j+1}}{c_j} = \frac{(1 - Q^{-n}Q^j)(1 + Q^{-n}Q^j)}{1 - Q^{j+1}} Q^{-j} (y Q^{2n}).$$ 
This expression is structurally identical to the ratio $t_{j+1}/t_j$ of the $_2\phi_0$ series. By equating the parameters, we find 
	$$ a = Q^{-n} = (-q)^{-n}, \qquad b = -Q^{-n} = -(-q)^{-n} \qquad \text{and} \qquad z = -y Q^{2n}.$$ 
Since $n = 2m$ is even, the signs in the exponents vanish, yielding 
	$$ a = q^{-n}, \qquad b = -q^{-n} \qquad \text{ and } \qquad Q^{2n} = q^{2n}.$$ 
 Thus, the argument becomes $z = -y q^{2n}$. 
Because the upper parameter $a = q^{-n}$ forces the $Q$-Pochhammer symbol $(q^{-n}; -q)_j$ to vanish for all $j > n$, the infinite series naturally terminates, perfectly matching our polynomial $F(y)$ of degree $n$, i.e. 
	$$ F(y) = \ _2\phi_0 \Big( \begin{matrix} q^{-n}, -q^{-n} \\ - \end{matrix} ; -q, -y q^{2n} \Big).$$ 
Finally, we apply the relation linking $F(y)$ to our counting problem:$$N_m(s,q) = q^{-n^2+2ns} F((-q)^{-s})$$Evaluating $F(y)$ at $y = (-q)^{-s}$ yields the series argument $z = -(-q)^{-s} q^{2n} = -(-q)^{2n-s}$. Substituting this back into the hypergeometric representation concludes the proof.
\end{proof}

We conclude this work with some remarks.
\begin{rem}
Let $Q=-q$ and $n=2m$. For the basic hypergeometric series
$$
H_{a,b}(z)={}_2\phi_0\!\left(
\begin{smallmatrix}a,b\\-\end{smallmatrix};Q,z\right),
$$
comparison of coefficients gives the contiguous relation
$$ H_{a,b}(z)-H_{a,b}(Qz) = -(1-a)(1-b)z\,H_{aQ,bQ}(\tfrac zQ). $$
This identity may be interpreted formally; the series used below terminate.
To recover the recurrence in Theorem~\ref{thm: recursive Nmsq} without shifting the upper
parameters, put
$$
H(z)=H_{Q^{-n},-Q^{-n}}(z).
$$
Writing $H(z)=\sum_{j=0}^{n}h_jz^j$, its coefficients satisfy
$$
(1-Q^j)h_j
=\bigl(Q^{j-1-2n}-Q^{1-j}\bigr)h_{j-1}
\qquad (1\leq j\leq n).
$$
The same identity holds for $j=n+1$ upon setting $h_{n+1}=0$.
Consequently,
$$
H(z)-(1+Q^{-2n}z)H(Qz)+zH(\tfrac zQ)=0.
$$
By Theorem~\ref{prop: hipergeometrica}, $F(y)=H(-yQ^{2n})$, and hence
$$
F(y)-(1-y)F(Qy)=yQ^{2n}F(\tfrac yQ).
$$
Taking $y=Q^{-(s-1)}$ and using
$ F(Q^{-k})=q^{n^2-2nk}N_m(k,q)$, 
we obtain, for $s\geq2$,
$$
N_m(s,q)=Q^{s-1}N_m(s-1,q)
+q^{4m}(1-Q^{s-1})N_m(s-2,q),
$$
which is precisely the recurrence in Theorem~\ref{thm: recursive Nmsq}.
\end{rem}

\section*{Final Remarks}
In this work, we have addressed the exact enumeration of solutions for systems of diagonal equations over finite fields, with an arbitrary number of equations and variables. We achieved this by translating the algebraic intersection problem to diagonal generalized Paley graphs. Thus the number of solutions can be given in terms of the spectum of these graphs. 

This spectral graph-theoretic approach stands in sharp contrast to standard methodologies in the literature. Classical techniques typically rely on the direct manipulation of multidimensional Jacobi and Gauss sums, or depend heavily on the invariant theory of quadratic forms. Such traditional methods often become analytically intractable for systems with $m \ge 2$ equations due to severe character sum entanglement, which usually restricts exact results to very specific degrees or small dimensions.

\enlargethispage{2cm}


\begin{thebibliography}{XXX}	
	\bibitem{AAR}
	\textsc{G.\@ E.\@ Andrews,R.\@ Askey, R.\@ Roy}. 
	\textit{Special functions}. 
	Encyclopedia of Mathematics and Its Applications. 71. Cambridge: Cambridge University Press. xvi, 664 p. (1999).		
	
	\bibitem{Bass1992}
	\textsc{H.\@ Bass}
	\textit{The Ihara-Selberg zeta function of a tree lattice}. 
	Int.\@ J.\@ Math.\@ \textbf{3:6}, (1992) 717--797.
	
	\bibitem{BCN}
	\textsc{A.E.\@ Brouwer, A.M.\@ Cohen, A.\@ Neumaier}. 
	\textit{Distance-Regular Graphs}. 
	Ergebnisse der Mathematik und ihrer Grenzgebiete. Berlin: Springer-Verlag, 1989, vol.\@ 18.
		
	\bibitem{CCG}{\sc C.\@ Cao, W-S.\@ Chou, J.\@ Gu}.
	\textit{On the number of solutions of certain diagonal equations over finite fields}.
	Finite Fields and App.\@ \textbf{42}, (2016) 225--252.	

	\bibitem{Carlitz1953}
	\textsc{L.\@ Carlitz}. 
	\textit{A note on multiple exponential sums}. 
	Pacific J.\@ Math.\@ \textbf{3}, (1953) 673--690.
	
	\bibitem{LN} {\sc  R.\@ Lidl, H.\@ Niederreiter}. 
	\textit{Finite fields}. \textit{Encyclopedia of Mathematics and its Applications}, 20. 
	Addison-Wesley Publishing Company, Advanced Book Program, Reading, MA, 1983.
		
	\bibitem{LP}
	{\sc T.\@ K.\@ Lim, C.\@ Praeger}. 
	\textit{On Generalized Paley Graphs and their automorphism groups}.
	Michigan Math.\@ J.\@ \textbf{58}, (2009) 294--308.
	
	\bibitem{LZ}
	{\sc X.\@ Liu, S.\@ Zhou}. 
	\textit{Eigenvalues of Cayley Graphs}. 
	Electron.\@ J.\@ Combin.\@ \textbf{29:2} (2022).
		
		
	\bibitem{MC2} 
	\textsc{O.\@ Moreno, F.\@ N.\@ Castro}. 
	\textit{Divisibility properties for covering radius of certain cyclic codes}.
	IEEE Trans.\@ Inf.\@ Theory \textbf{49}, (2003) 3299--3303.
			
	\bibitem{MP}	
	{\sc G.\@ L.\@ Mullen, D.\@ Panario}. 
	\textit{Handbook of finite fields}. 
	CRC Press, 2013.
		
	\bibitem{Oliveira2026}
	\textsc{J.\@ A.\@ Oliveira}. 
	\textit{On diagonal equations over finite fields}. \textit{arXiv preprint arXiv:2008.12232v2}, (2026).

	\bibitem{PP} 
	\textsc{G.\@ Pearce, C.E.\@ Praeger}. 
	\textit{Generalized Paley graphs with a product structure}.
	Ann.\@ Comb.\@ \textbf{23}, (2019) 171--182.
		
	\bibitem{PPri} 
	\textsc{M.\@ Pérez, M.\@ Privitelli}.
	\textit{Estimates on the number of $\ff_q$-rational solutions of variants of diagonal equations over finite fields}. 
	Finite Fields Appl.\@ \textbf{68}, Article ID 101728, 30 p. (2020).	
		
	\bibitem{PPri2} 
	\textsc{M.\@ Pérez, M.\@ Privitelli}.
	\textit{On the number of solutions of systems of certain diagonal equations over finite fields}.
	J.\@ Number Theory \textbf{236}, (2022) 160--187.
			
	\bibitem{PV6} 
	\textsc{R.A.\@ Podest\'a, D.E.\@ Videla}.
	\textit{The Waring's problem over finite fields through generalized Paley graphs}.
	Discrete Math.\@ \textbf{344}, (2021) 112324.
	
	\bibitem{PV7} 
	\textsc{R.A.\@ Podest\'a, D.E.\@ Videla}.
	\textit{A reduction formula for Waring numbers through generalized Paley graphs}.
	J.\@ Algebr.\@ Comb.\@ \textbf{56:4}, (2022), 1255--1285.
	
	\bibitem{PV4} 
	\textsc{R.A.\@ Podest\'a, D.E.\@ Videla}.
	\textit{The weight distribution of irreducible cyclic codes associated with decomposable generalized Paley graphs}.
	Adv.\@ Math.\@ Comm.\@ \textbf{17:2}, (2023) 446--464. 
	
	\bibitem{PV3}
	\textsc{R.A.\@ Podest\'a, D.E.\@ Videla}.
	\textit{Generalized Paley graphs equienergetic with their complements}.
	Linear Multilinear Algebra \textbf{72:3}, (2024) 488--515.  
	
	\bibitem{PV1}	
	\textsc{R.A.\@ Podest\'a, D.E.\@ Videla}.
	\textit{Spectral properties of generalized Paley graphs of $(q^\ell+1)$-th powers and applications}.
	Discrete Math.\@ Algorithms Appl.\@ \textbf{17:4}, (2025) 2450056 
	
	\bibitem{PV8} 
	\textsc{R.A.\@ Podest\'a, D.E.\@ Videla}.
	\textit{Spectral properties of generalized Paley graphs}. 
	Australas.\@ J.\@ Comb.\@ \textbf{91:3}, (2025) 326--365. 
	
	\bibitem{PV18} 
	\textsc{Ricardo A.\@ Podestá, Denis E.\@ Videla}.
	\textit{Connected components and non-bipartiteness of generalized Paley graphs}. 
	Ann.\@ Comb.\@ \textbf{29}, (1235--1259), 2025.
	
	\bibitem{PV19} 
	\textsc{R.A.\@ Podest\'a, D.E.\@ Videla}.
	\textit{The nature of the spectrum of generalized Paley graphs and weak Waring numbers over finite fields}. 
	(2026), \href{https://doi.org/10.48550/arXiv.2604.06513}{arXiv.2604.06513}.
	
	\bibitem{PV26} 
	\textsc{R.A.\@ Podest\'a, D.E.\@ Videla}.
	\textit{On the number of solutions of one diagonal equation through generalized Paley graphs}. 
	Work in progress.
	
	\bibitem{St} 
	\textsc{D.\@ Stanton}. 
	\textit{Three addition theorems for some $q$-Krawtchouk polynomials}. 
	Geometriae Dedicata, \textbf{10:1-4}, (1981) 403--425. 
	
	\bibitem{Sunada1986}
	\textsc{T.\@ Sunada}.
	\textit{$L$-functions in geometry and some applications}. Curvature and topology of Riemannian manifolds, Proc. 17th Int. Taniguchi Symp., Katata/Jap. 1985, 
	Lect.\@ Notes Math.\@ \textbf{1201}, (1986) 266--284.

	\bibitem{Terras2010}
	\textsc{A.\@ Terras}.
	\textit{Zeta functions of graphs. A stroll through the garden}. 
	Cambridge Studies in Advanced Mathematics 128, Cambridge University Press 2010.
	
	\bibitem{Tietavainen1965}
	\textsc{A. Tietäväinen}. 
	\textit{On the number of solutions of some equations and systems of equations in a finite field}. 
	Ann.\@ Univ.\@ Turku.\@ \textbf{84}, (1965) 1--10.
	
	\bibitem{V} 
	\textsc{D.E.\@ Videla}.
	\textit{On diagonal equations over finite fields via walks in NEPS of graphs}.
	Finite Fields App.\@ \textbf{75}, (2021) 101882. 
	
	\bibitem{W}
	{\sc A.\@ Weil}.
	\textit{Numbers of solutions of equations in finite fields}.
	Bull.\@ Am.\@ Math.\@ Soc.\@ \textbf{55}, (1949) 497--508.
		
	\bibitem{Yip}
	\textsc{Yip, Chi Hoi}
	\textit{On the directions determined by Cartesian products and the clique number of generalized Paley graphs.} 
	Integers \textbf{21}, Paper A51, 31 p. (2021).	
	
	\bibitem{ZHJYC}	
	{\sc X.\@ Zeng, L.\@ Hu, W.\@ Jiang, Q.\@ Yue, X.\@ Cao}. 
	\textit{The weight distribution of a class of p-ary cyclic codes}.
	Finite Fields Appl.\@ \textbf{16:1}, (2010) 56--73.
		
	\bibitem{ZWZH} 
	{\sc D.\@ Zheng, X.\@ Wang, X.\@ Zeng, L.\@ Hu}.
	\textit{The weight distribution of a family of p-ary cyclic codes}. 
	Des.\@ Codes Cryptogr.\@ \textbf{75}, (2015) 263--275.
		
	\bibitem{ZZDX2} 
	{\sc  Z.\@ Zhou, A.\@ Zhang, C.\@ Ding, M.\@ Xiong}. 
	\textit{The weight enumerator of three families of cyclic codes}. 
	IEEE Trans.\@ Inform.\@ Theory \textbf{59:9}, (2013) 6002--6009.
\end{thebibliography}
\end{document}